\documentclass[12pt]{article}

\usepackage[margin=2.5cm]{geometry}

\usepackage{amsmath,amsthm,amsfonts,amscd,amssymb,bbm,mathrsfs,enumerate,url}

\usepackage{graphicx,tikz}
\usepackage[pdfborder={0 0 0}]{hyperref}
\usetikzlibrary{matrix,arrows}
\usetikzlibrary{decorations.pathreplacing}
\usetikzlibrary{decorations.pathmorphing}
\usetikzlibrary{patterns,fadings}

\def\RR{{\mathbb R}}
\def\CC{{\mathbb C}}
\def\NN{{\mathbb N}}
\def\ZZ{{\mathbb Z}}

\def\a{\alpha}

\def\k{\kappa}

\def\l{\lambda}
\def\cW{{\cal W}}

\def\L{\Lambda}

\def\1{{\mathbbm 1}}

\def\uone{{\rm U(1)}}

\def\diff{{\rm Diff}}

\def\diffs1{\diff(S^1)}

\def\psl2r{{\rm PSL}(2,\RR)}
\def\sl2r{{\rm SL}(2,\RR)}
\def\su11{{\rm SU}(1,1)}
\def\2dmob{{\overline{\psl2r}\times\overline{\psl2r}}}
\def\<{\langle}
\def\>{\rangle}

\def\cl{\colon}

\newcommand{\Jone}{J_{[1]}}
\newcommand{\Jtwo}{J_{[2]}}
\newcommand{\aone}{a_{[1]}}
\newcommand{\atwo}{a_{[2]}}

\newcommand{\Ttwozero}{T_{[2],0}}
\newcommand{\Lone}{L_{[1]}}
\newcommand{\Ltwo}{L_{[2]}}

\newtheorem{theorem}{Theorem}[section]

\newtheorem{corollary}[theorem]{Corollary}
\newtheorem{proposition}[theorem]{Proposition}
\newtheorem{lemma}[theorem]{Lemma}
\theoremstyle{remark}
\newtheorem{remark}[theorem]{Remark}

\title{Unitary representations of the $\cW_3$-algebra with $c\geq 2$}
\date{} 
\author{
{\bf Sebastiano Carpi}\footnote{Supported in part by ERC advanced grant 669240 QUEST ``Quantum Algebraic Structures and Models'' and GNAMPA-INDAM.},\\
   Dipartimento di Matematica, Universit\`a di Roma Tor Vergata\\
   Via della Ricerca Scientifica 1, I-00133 Roma, Italy\\
   email: {\tt carpi@mat.uniroma2.it}\\
\\
{\bf Yoh Tanimoto}\footnote{Supported by Programma per giovani ricercatori, anno 2014 ``Rita Levi Montalcini''
of the Italian Ministry of Education, University and Research.
}
\\
   Dipartimento di Matematica, Universit\`a di Roma Tor Vergata\\
   Via della Ricerca Scientifica 1, I-00133 Roma, Italy\\
   email: {\tt hoyt@mat.uniroma2.it}\\
\\
{\bf Mih\'aly Weiner}\footnote{Supported in part by the Bolyai J\'anos and Bolyai+ scholarships,
by the NRDI grants K 124152, KH 129601, K 132097 and
by the ERC advanced grant 669240 QUEST ``Quantum Algebraic Structures and Models''.}\\
MTA-BME Lend\"ulet Quantum Information Theory Research Group \\
Budapest University of Technology \& Economics (BME)\\
M\H{u}egyetem rkp. 3-9, H-1111 Budapest, Hungary\\
email: {\tt mweiner@math.bme.hu}
}
\begin{document}
\maketitle
\begin{abstract}
We prove unitarity of the vacuum representation of the $\cW_3$-algebra for all values of the central charge $c\geq 2$.
We do it by modifying the free field realization of Fateev and Zamolodchikov resulting
in a representation which, 
by a nontrivial argument, can be shown to be unitary on a certain
invariant subspace, although it is not unitary on the full space
of the two currents needed for the construction.
These vacuum representations give rise to simple unitary vertex operator algebras.
We also construct explicitly unitary representations for many positive lowest weight values.
Taking into account the known form of the Kac determinants, we then completely clarify the question of unitarity
of the irreducible lowest weight representations of the $\cW_3$-algebra in the $2\leq c\leq 98$ region.

\end{abstract}

\section{Introduction}
The $\cW_N$ ($N=2,3,\cdots$) algebras are ``higher spin'' extensions of the Virasoro algebra \cite{Zamolodchikov85, FZ87, FL88},
with $\cW_2$ being the Virasoro algebra itself and $\cW_3$ in some sense the simplest one \emph{without} a Lie algebra structure.
For general $N$, the $\cW_N$-algebra is generated by $N-1$ fields, the first one of which
is the Virasoro field.
For some discrete values of the central charge $c<N-1$, 
they have been recently realized as a certain coset, showing unitarity of their vacuum representations
(i.e.\! the irreducible representations with zero lowest weights)
as well as many other representations \cite{ACL19}. In the Virasoro
case ($N=2$), this is the famous construction of Goddard, Kent and Olive \cite{GKO86} and the corresponding central charge values are
\[
 c=1-\frac{6}{m(m+1)}\;\;\; m=3,4,5\ldots
\]
whereas for the $\cW_3$-algebra,
these values are \cite{Mizoguchi89, Mizoguchi91}
\[
 c = 2\left(1-\frac{12}{m(m+1)} \right)\;\;\; m=4,5,6\ldots
\]
and in both cases $N=2,3$ it is known that there are no other unitary representations
in the $c<N-1$ region than the ones obtained in this manner.
Though this coset realization is recently generalized \cite{ACL19} to an even wider class of $\cW$-algebras, it is not expected to take us above the central charge value $c=N-1$, where rationality cannot hold.
Indeed, as far as we know, unitarity has never been shown for any central charge value $c>N-1\ge 2$.
Note that unlike in the Virasoro (or in the affine Kac-Moody) case, when $N\geq 3$ -- because of the lack of a Lie algebra structure --
one cannot simply produce representations of $\mathcal W_N$ by e.g.\! taking tensor products of known ones. Because of the difficulty of finding explicitly unitary constructions, some even expected the
$\cW_N$-algebras to not to have
unitary vacuum representations
for $c>N-1\ge 2$ (see e.g.\! \cite{ACHMP18}).
In this paper, we prove in fact that the vacuum representation of the $\cW_3$-algebra
is unitary for \emph{any} value of the central charge $c\ge 2$.

In the Virasoro case, unitarity for $c>1$ can be settled using Kac determinants
see e.g.\! \cite[Section 8.4]{KR87}.
At any ``energy level'' (i.e.\! eigenspace of the conformal Hamiltonian), the Kac determinant is a polynomial of the central charge $c$ and lowest weight $h$. Since all Kac determinants are strictly positive in the region $\{(c,h):c>1,h>0\}$,
by a continuity argument, unitarity in a single case inside that region
(which can be easily obtained e.g.\! by taking tensor products) implies
unitarity for the whole closure $\{(c,h):c \ge 1,h \ge 0\}$.
In case of the $\cW_3$-algebra, the difficulty is twofold.
First, one cannot obtain unitary representations with $c>2$
by tensor product. Second, the Kac determinants -- which are this time rational functions of the central charge $c$
and lowest weights $h, w$ and are explicitly worked out in \cite{Mizoguchi89} by Mizoguchi
-- show that when $c>2$, no irreducible lowest weight representation can be unitary in a neighbourhood of $h=w=0$
(apart from the vacuum itself). Hence the physically most important representation, the vacuum one, cannot be accessed in this manner from the $(h,w)\neq (0,0)$ region. With the usual indirect method ruled out,
we are lead to consider unitarity in a more constructive approach.

The explicit construction of unitary vacuum representations in the $c>N-1$ region is not trivial even in the Virasoro ($N=2$) case.
Buchholz and Schulz-Mirbach \cite{BS90} provided an interesting construction in this regard. They first realized the Virasoro algebra with central charge $c>1$ with the help of the $\uone$-current (a field whose Fourier modes form a representation of the Heisenberg algebra) in a -- strictly speaking --
non-unitary way. These representations (which we simply call the BS-M construction) turn out to be
``almost unitary'': the only problem is a singularity at just one point
(indeed, they only needed their construction to be defined on the punctured circle). As observed in \cite{Weiner08}, the BS-M construction may be viewed as a non-unitary representation of the Virasoro algebra admitting an invariant subspace containing the vacuum vector $\Omega$, on which it \emph{is} unitary. Inspired by the BS-M construction and the mentioned observation,
we start with a pair of commuting $\uone$-currents in their unitary
vacuum representation and modify them so that the
Fateev-Zamolodchikov free field realization of the $\cW_3$-algebra \cite{FZ87}
associated with this modified representation of the Heisenberg algebra
gives a stress-energy field corresponding to the BS-M one.
Similarly to the BS-M case, the obtained new stress-energy and $W(z)$ fields will not give a unitary representation of the $\cW_3$-algebra on the full
space but they become so on a subspace generated by $\Omega$. However, the proof of this relies on a rather involved argument exploiting the degeneracy of the vacuum representation:
the same construction with nonzero lowest weights does \emph{not} have unitarity on the minimal invariant
subspace containing the lowest weight vector. 

Whereas unitarity of the vacuum is difficult to treat, it turns out that some \emph{non}-vacuum representations can be shown to be unitary in a relatively 
simple, constructive manner.
Making another suitable use of the realization of Fateev and
Zamolodchikov, we obtain a manifestly unitary representation of the $\cW_3$-algebra on a full unitary
representation space of two $\uone$-currents.
In this way,
we produce unitary representations with $h\geq \frac{c-2}{24}\geq 0$ and $w$ limited in a certain interval depending on $c$ and $h$.
This is similar to the Virasoro
case, where an oscillator representation with a modified Sugawara construction
gives manifestly unitary representations
for all $h\geq \frac{c-1}{12}\geq 0$;
see e.g.\! \cite[Section 3.4]{KR87}.

Having already found some unitary representations, one can use the known form of the Kac determinant to arrive at even further values of $c,h$ and $w$. In this way, for $2\leq c\leq 98$ we completely clear
the question of unitarity. When $c>98$, determining the sign of the Kac determinant becomes harder; our results there remain partial.

This paper is organized as follows.
In Section \ref{preliminaries} we give a
summary of formal series with operator coefficients on Hermitian
vector spaces and on the $\cW_3$-algebra, the current algebras and their
representations. 
Apart from self-containment, we use the occasion to fix notations and conventions. An important tool for unitarity, the Kac determinant, is also introduced. 
Our main results are in Section \ref{unitarity}, where we prove the unitarity of various representations of the $\cW_3$-algebra and completely classify unitary lowest weight representations with central charge 
$c \in [2,98]$.
We also briefly explain in a remark how each unitary vacuum
representation gives rise to a simple unitary vertex operator
algebra.
Finally, in Section \ref{outlook} we collect possible future directions and open problems.

The non-constructive part of our work (where we exploit Kac determinants) is based on the existence of lowest
weight representations with invariant forms. Yet, as the $\cW_3$-algebra is not a Lie algebra,
the existence of lowest weight representations with invariant forms for all values of lowest weights is not straightforward. Though implicit in the literature, we could not find a reference suitable for our needs, so we added an 
Appendix \ref{verma} to our work where we clarify this issue by a novel, analytic method.

\section{Preliminaries}\label{preliminaries}
\subsection{Formal  series and fields}\label{formal}
Let $V$ be a vector space and $A_n:V \to V (n \in \ZZ)$ be a sequence of linear operators acting on $V$. We say that the formal  series $A(z)=\sum_{n\in \ZZ}A_{n} z^{-n}$ is a {\bf field} on $V$ if for every $v\in V$,
there is $n_v$ such that $A_n v=0$ whenever $n\ge n_v$. We 
shall refer to the operators $\{A_n\}_{n\in \ZZ}$ as the {\bf Fourier modes} of $A(z)$.

The (formal) derivative of $A(z)=\sum_{n\in \ZZ}A_n z^{-n}$ is $\partial_z A(z) = \sum_{n\in \ZZ} (-n)A_{n}z^{-n-1}$; note that if $A(z)$ is a field, so is $\partial A(z)$.
When $A(z), B(\zeta)$ are two formal  series, the product $A(z)B(\zeta)$ is a formal  series
in two variables $z,\zeta$ and we shall use the notations $\partial_\zeta, \partial_z$ in the obvious way. Moreover, we shall also use the notation
\begin{align}\label{eq:derivative}
A'(z) := iz\partial_z A(z) = \sum_{n\in \ZZ} (-in) A_n z^{-n},
\end{align}
which we call the ``derivative along the circle''.

Although the product of two formal  series of the same variables does not make sense in general,
there are some pairs of formal  series that can be multiplied. For example, the product of a formal series in variables $z$ and $\zeta$ of the form $B(z/\zeta)$ with any other formal  series in \emph{either} $z$ \emph{or} $\zeta$ (but not in both!) makes sense. In particular, the product 
$\delta(z-\zeta)A(\zeta)$, where 
 $\delta(z-\zeta) := z^{-1}\sum_n (\frac \zeta z)^n$ is the {\bf formal delta function}, is well-defined; see more at \cite[Section 2.1]{Kac98}.
Also, if $B(z)=\sum_{n\in \ZZ}B_n z^{-n}$ is a field then an infinite sum of the form
\[
\sum_{n\geq N} A_{k-n}B_n 
\]
(where $N,k\in \ZZ$) becomes finite on every vector and hence gives rise to a well-defined linear map. In particular, every field can be multiplied with a formal  series of the form $\sum_{n\leq N}c_n z^{-n}$ (where the coefficients $c_n$ may be scalars or themselves linear maps). It then turns out that if 
$F(z)=\sum_{n\in \ZZ}F_{(n)}z^{-n-1}$ and $G(\zeta)$ are fields, then by setting
$F_+(z) := \sum_{n < 0} F_{(n)}z^{-n-1}, F_-(z) := \sum_{n \ge 0} F_{(n)}z^{-n-1}$, the {\bf normally ordered product}
\[
: F(z)G(\zeta): = F_+(z)G(\zeta) +  G(\zeta)F_-(z)
\]
is well-defined even at $z=\zeta$ (i.e. after replacing $\zeta$ by $z$)
and the obtained formal series $\cl F(z)G(z)\cl$ is again a field, see e.g.\! \cite[Section 3.1]{Kac98}. 
If $F(z)$ and $G(\zeta)$ commute, so do $F_\pm(z)$ and $G(\zeta)$, therefore,
$\cl F(z)G(\zeta)\cl = F(z)G(\zeta)$. Note that in general the normal product of fields is neither commutative nor associative; in particular, to have an unambiguous meaning, we need
to specify what we mean by the normal power $\cl F(z)^n\cl$.
Following the standard conventions, we define the $n$-th power in a recursive manner
by the formula $\cl F(z)^n\cl \, = \, \cl F(z)(\cl F(z)^{n-1}\cl)\cl$, and more in general,
$\cl F_1(z)F_2(z)\cdots F_n(z)\cl \, = \, \cl F_1(z)(\cl F_2(z)\cdots F_n(z)\cl)\cl$.

\subsection{Formal adjoints of formal  series and fields}
Let $V$ be a $\CC$-linear space equipped 
with a Hermitian form $\<\cdot, \cdot\>$ (i.e.\! a self-adjoint sesquilinear form) and 
$A,B: V\to V$ linear operators. If
\begin{align}
\label{formal_adjoint}
 \<B\Psi_1, \Psi_2\> = \<\Psi_1, A\Psi_2\>, \quad \text{ for all } \Psi_1, \Psi_2 \in V,   
\end{align}
then we say that $A$ and $B$ are {\bf  adjoints} of each other and with some abuse of notation we write $B=A^\dagger$. Note however, the following: 1) such an $A^\dagger$ might not exist,
2) when $\<\cdot, \cdot\>$ is degenerate, $A^\dagger$ may not even be unique.
Nevertheless, for any two operators
$A,B$ the statement $B=A^\dagger$ is unambiguous: it simply means that they
satisfy equation \eqref{formal_adjoint}. We also say that $A$ is
{\bf symmetric}\footnote{We use ``symmetric'' instead of ``self-adjoint'' in order to avoid confusions
with the notion of self-adjoint operator on a Hilbert space in view of a possible Hilbert space completion
of the vector space $V$. This ``symmetry'' has nothing to do with symmetric operators with respect to
a \emph{bilinear} (instead of sesquilinear) form.} when $A=A^\dagger$.

We define the adjoint of the formal  series $A(z)=\sum_{n\in \ZZ}A_n z^{-n}$ to be the formal series 
\[
A(z)^\dagger := \sum_{n\in \ZZ}A^\dagger_n z^{n},
\]
i.e.\! we treat the variable $z$ as if it were a complex number in $S^1:= \{z\in \CC : \, |z|=1 \}$.
As a direct consequence of our definition, $A(z)$ is {\bf symmetric} -- that is, $A(z)^\dagger = A(z)$ as 
formal series -- if and only if $A_n^\dagger = A_{-n}$ for all $n\in\ZZ$.
Moreover, if $A(z)$ is symmetric, then so is its circle derivative $A'(z)$ of \eqref{eq:derivative}:
this is exactly why we shall prefer it to $\partial_z A(z)$.
Note that this is also the convention found in the paper \cite{BS90} of Buchholz and Schulz-Mirbach.

If $f(z)$ is a {\bf trigonometric polynomial}, i.e.\! a finite series $f(z)=\sum_{|n|<N }c_n z^{-n}$, and $A(z)$ is a symmetric field, then one finds that
\[
(f(z)A(z))^\dagger = \overline{f}(z) A(z)\;\;\;
\textrm{where}\;\;
\overline{f}(z):= \sum_{|n| < N}\overline{c_{-n}} z^{-n}.
\]
In particular, if $\overline{c_n}=c_{-n}$ for all $n$ -- or
equivalently: if $f$ takes only real values on $S^1$ -- then $f(z)A(z)$ is symmetric. This is not surprising at all; in fact, more in general, one has that if $A(z)$ and $B(z)$ are commuting symmetric fields, then their product
$A(z)B(z)$ is also a symmetric field.
However, in this paper we shall often consider expressions of the type $\rho(z)A(z), \rho'(z)A(z)$ where $\rho(z) = -i\frac{z-1}{z+1}$.
In order to give an unambiguous meaning\footnote{From the point of view of quantum field theory (cf.\! \cite{BS90}),
$\rho$ should be regarded as a \emph{function} on $S^1 \setminus \{-1\}$ rather than a \emph{formal series};
depending on the choice of region, it has different expansions.}
to the expression $\rho(z)A(z)$, we take the expansion around $z=0$, where it holds that
\begin{align}\label{eq:rho}
 \rho(z) = -i\frac{z-1}{z+1}
= -i(z-1)\sum_{n\ge 0} (-1)^n z^n = i\left(1 + 2\sum_{n\ge 1} (-1)^n z^n\right)=: \sum_n \rho_n z^{-n}.
\end{align}
Accordingly we regard $\rho(z)$ as a field (note that $\rho_n = 0$ for $n > 0$), and since it is scalar valued, it commutes with anything and  its product with another field $A(z)$ is meaningful without need of normal ordering: $\rho(z)A(z) = \sum_{n,k} (\rho_k A_{n-k})z^{-n}$.
Similarly, the product $\rho'(z)A(z)$, with $\rho'(z)$ given by \eqref{eq:derivative},
is defined as a field.

Although $\rho(z)$ is not defined at $z=-1$ as a function (it has a singularity there),
it takes only real values on the punctured circle $S^1\setminus \{-1\}$ and hence so does its circle 
derivative $\rho'(z)$. So one might wonder whether $\rho(z)A(z)$ and $\rho'(z)A(z)$ are still symmetric if $A(z)$ is a symmetric field. A quick check reveals that the answer in general is \emph{negative}:
the problem is caused by the non-symmetric expansion \eqref{eq:rho}.
But if $r(z)$ is a trigonometric polynomial and $r(-1)=0$, then the singularity of $r(z)\rho(z)$ at $z=-1$ is removable. Actually, it is clear that
in this case $r(z)=(z+1)t(z)$ where $t$ is another trigonometric polynomial and hence 
$s(z)=r(z)\rho(z) = -i\frac{z-1}{z+1} (z+1) t(z) = -i(z-1)t(z)$ is also a trigonometric polynomial
for which $\overline{s}(z)= \overline{r}(z)\rho(z)$. Hence in this case 
\[
(r(z)\rho(z)A(z))^\dagger = \overline{r}(z) \rho(z)A(z),
\]
as if $\rho(z)A(z)$ were symmetric. If further $r'(-1)=0$, then 
also the singularity of $r(z)\rho'(z)$ will be removable, resulting in
$(r(z)\rho'(z)A(z))^\dagger = \overline{r}(z) \rho'(z)A(z)$.
These observation will become important in the proof of unitarity of vacuum representations.     

\subsection{The \texorpdfstring{$\cW_3$}{W3}-algebra}\label{w3}
For our purposes the $\cW_3$-algebra (see \cite{BS93, Artamonov16} for reviews) at central charge $c\in \CC$, $c\neq -\frac{22}{5}$, consists of two fields $L(z)=\sum_{n\in\ZZ}L_n z^{-n-2}$ and $W(z)=\sum_{n\in\ZZ}W_n z^{-n-3}$ acting on a $\CC$-linear space $V$ such that
\begin{align}\label{eq:commfield}
 [L(z), L(\zeta)] &= \delta(z-\zeta)\partial_\zeta L(\zeta) + 2\partial_\zeta \delta(z-\zeta) L(\zeta)
 + \frac c{12}\partial^3_\zeta \delta(z-\zeta), \nonumber \\
 [L(z), W(\zeta)] &= 3 \partial_\zeta \delta(z-\zeta) W(\zeta) + \delta(z-\zeta)\partial_\zeta W(\zeta), \nonumber\\
 [W(z), W(\zeta)] &= \frac c{3\cdot 5!} \partial_\zeta^5 \delta(z-\zeta) + \frac13 \partial_\zeta^3\delta(z-\zeta) L(\zeta)
  + \frac12 \partial_\zeta^2\delta(z-\zeta)\partial L(\zeta) \nonumber \\
  &\;\;+ \partial_\zeta \delta(z-\zeta) \left(\frac 3{10} \partial_\zeta^2 L(\zeta) + 2b^2 \Lambda(\zeta) \right)
  + \delta(z-\zeta) \left(\frac1{15}\partial_\zeta^3 L(\zeta) + b^2 \partial_\zeta \Lambda(\zeta)\right)
\end{align}
where $b^2 = \frac{16}{22+5c}$ and $\L(z)=\;:L(z)^2: -
 \frac3{10} \partial_z^2 L(z)$.
Equivalently, in terms of Fourier modes 
the requirements read
\begin{align}
 [L_m, L_n] &= (m-n)L_{m+n} + \frac{c}{12}m(m^2-1)\delta_{m+n,0}, \nonumber\\
 [L_m, W_n] &= (2m-n)W_{m+n}, \nonumber \\
 [W_m, W_n] &= \frac{c}{3\cdot 5!}(m^2-4)(m^2-1)m\delta_{m+n,0} \nonumber \\
            &\quad + b^2(m-n)\L_{m+n} + \left[\frac1{20}(m-n)(2m^2-mn+2n^2-8))\right]L_{m+n}, \label{eq:comm}
\end{align}
where again $b^2 = \frac{16}{22+5c}$ and $\L_n = \sum_{k>-2} L_{n-k} L_k + \sum_{k\le-2} L_k L_{n-k} - \frac3{10}(n+2)(n+3)L_n$.
The first of these commutation relations says that the operators $\{L_n\}_{n\in\ZZ}$ form a representation of the {\bf Virasoro algebra} and consequently, we shall say that $L(z)$ is a {\bf Virasoro} (or alternatively: a {\bf stress-energy}) {\bf field}. 

Note that one cannot consider (\ref{eq:comm}) (together with the definitions of $b$ and $\L_n$)
as the defining relations of an associative algebra (as it is sometimes loosely stated in the literature), since the infinite sum appearing in $\L_n$ does not have an \emph{a priori} meaning: it makes sense if $\{L_n\}$ form a field on $V$.
Under the term ``$\cW_3$-algebra'', one studies \emph{general properties} that hold for operators $\{L_n, W_n\}_{n\in \ZZ}$ satisfying the above relations. On the other hand, a concrete realization  on a linear space is referred to as a {\bf representation}, although we do not define here
an associative algebra called the $\cW_3$-algebra. A universal object with these relations can be defined in the context of vertex operator algebras \cite{DK05, DK06}; however, here we do not wish to follow that way.

We shall say that a Hermitian form $\<\cdot,\cdot\>$ is
{\bf invariant} for a representation of the $\cW_3$-algebra, 
if it makes the fields
\[
 T(z) := z^2 L(z)=\sum_{n} L_n z^{-n}, \quad M(z) := z^3W(z) =\sum_{n} W_n z^{-n} 
\]
symmetric. Equivalently, in terms
of Fourier modes, the requirement of invariance is that $L_n^\dagger = L_{-n}$ and $W_n^\dagger=W_{-n}$
for all $n\in\ZZ$. A representation together with an \emph{inner product} -- or as is also called: \emph{scalar product} -- 
(i.e.\! a positive definite Hermitian form) is said to be {\bf unitary}. 

Note that while
in papers concerned with vertex operator algebras, the Virasoro field is typically
denoted by $L(z)$ (as in our work), physicists often use $T(z)$ for the same object. Here we chose to reserve this symbol for the ``shifted'' field $T(z)=z^2 L(z)$ in part to follow the notations of \cite{BS90} used   
by Buchholz and Schulz-Mirbach, and in part simply because being interested by unitarity, we will
actually use more the combination $z^2 L(z)$ than $L(z)$ on its own.

\subsection{\texorpdfstring{The $\uone$-current (or Heisenberg) algebra}
{The U(1)-current (or Heisenberg) algebra}}\label{u1}
The $\uone$-current  (or Heisenberg) algebra is an infinite-dimensional Lie algebra spanned freely
by the elements $\{a_n\}_{n\in\ZZ}$ and a central element $Z$ with commutation relations
\begin{align}\label{eq:u(1)}
 [a_m, a_n] = m\delta_{m+n,0} Z.
\end{align}

We shall be only interested in representations of this algebra where $Z$ acts as the identity and the formal  series 
\[
 a(z) = \sum_{n\in\ZZ} a_n z^{-n-1}
\]
(where, by the usual abuse of notations, we denote 
the representing operators with the same symbol as
the abstract Lie algebra elements) is a field. Note that in many relevant works regarding the $\cW_3$-algebra and published in physics journals, this field appears as ``the derivative of the massless free field'' and is denoted by $\partial_z\varphi(z)$ (e.g.\! in \cite{FZ87} and in \cite{Mizoguchi89}), although in our sense, in general\footnote{Unless we are in a representation where $a_0=0$} there is no field $\varphi(z)$ whose derivative is $a(z)$. Note also that the commutation relation \eqref{eq:u(1)} with $Z:= 1$
is equivalent to 
\begin{equation}
\label{eq:aa}
[a(z), a(\zeta)] = \partial_\zeta \delta(z-\zeta).
\end{equation}

Suppose now that we are also given a Hermitian
form $\<\cdot,\cdot\>$ on our representation space. We say that it is
{\bf invariant} for our representation, if it makes 
\[
 J(z):= za(z) =\sum_{n\in\ZZ} a_n z^{-n}
\]
symmetric; this is equivalent to
the condition $a_n^\dagger=a_{-n}$ for all $n\in \ZZ$. A representation together with
an invariant inner product, i.e.\! an invariant \emph{positive definite} Hermitian form,
is said to be  {\bf unitary}.

Similarly to what we did for $J(z)$ and $a(z)$, we also introduce in general the ``shifted'' 
normal powers $:J^{n}: \!(z) = z^n : a(z)^n:$. Again, the reason for working with them
(rather than with the usual powers\footnote{Note that $:J^n:(z) = z^n :a(z)^n:$ is \emph{different} from the
$n$-th normal power $:J(z)^n:= :(za(z))^n:$.}) is symmetry:
given an invariant Hermitian form, it is this combination which becomes symmetric. For example, for $n=2$ we have
\[
 : J^{2}: \!(z)=z^2: a(z)^2:  = za_+(z)\cdot za(z) + za(z)\cdot za_-(z).
\]
Now $za_+(z) = \sum_{n < 0} a_n z^{-n}$ and hence $(za_+(z))^\dagger = za_-(z) - a_0$. Moreover,
as $a_0^\dagger = a_0$ commutes with all $a_n$, putting all together we have that
\[
: J^{2}: \!(z)^\dagger = za(z)\cdot (za_-(z) - a_0) + (za_-(z) + a_0)\cdot za(z) 
 = : J^{2}: \!(z).
\]
For higher powers, symmetry of $: J^{n}: \!(z)$ is justified in a similar manner.

If $a(z)=\sum_{n\in\ZZ} a_n z^{-n-1}$ is a field 
satisfying the commutation relation \eqref{eq:aa}, then
its {\bf associated} (or {\bf canonical}) 
stress-energy field is
\[
 L(z) = \sum_{n\in \ZZ}L_n z^{-n-2} = \frac12 : a(z)^2:.
\]
Its Fourier modes $L_n=\frac{1}{2}(\sum_{k > -1} a_{n-k}a_k + \sum_{k \le -1} a_k a_{n-k})$ form a representation of the Virasoro algebra with central charge $c=1$.
By elementary computations, $[L_n,a_m] = -m a_{n+m}$ and it then 
follows that for any $\eta,\kappa\in \CC$, the operators
\begin{align*}
L_n -i\kappa n a_n + \eta a_n \;\; (n\neq 0),\;\;\; 
L_0 + \eta a_0+ \frac{\kappa^2+\eta^2}{2},
\end{align*}
also form a representation of the Virasoro algebra with 
central charge $c(\eta,\kappa)=1+12\kappa^2$;
see e.g. \cite[Section 3.4]{KR87}.
Using circle derivatives, the corresponding ``shifted'' stress-energy field can be written as
\begin{align}
\label{constrKR87}
\frac12 : J^2:\!(z) +\kappa J'(z) + \eta J(z) + \frac{\kappa^2+\eta^2}{2}.
\end{align}

For the formal series $J(z)=za(z)=\sum_{n\in\ZZ}a_n z^{-n}$ where $a(z)$ satisfies (\ref{eq:aa}),
a nonzero vector $\Omega_q$ is said to be a {\bf lowest weight vector} with lowest weight $q\in\CC$ if 
\[
 \text{for all } m>0:\;a_m\Omega_q =0,\;\;\;\; a_0\Omega_q = q\Omega_q.
\]
If $\Omega_q$ is also cyclic, then the whole representation is said to be a {\bf lowest weight representation}. It turns out that for every $q\in\CC$, such a representation exists (up to equivalence) uniquely; this is the Verma module  $\underline{V}^\uone_q$. In this representation one has that vectors of the form
\[
 a_{-n_1}\cdots a_{-n_k}\Omega_q,
\]
where $1 \le n_1 \le \ldots \le n_k$, form a basis, the formal series $a(z)$ is a field and further that $a_0$ is the (multiplication by the) scalar $q$. Moreover, when $q \in \RR$, there exists a unique Hermitian form $\<\cdot,\cdot\>$ on
$\underline{V}^\uone_q$ with normalization
$\<\Omega_q,\Omega_q\>=1$, which is invariant for the representation
(the ``canonical Hermitian form''). This form is automatically positive definite,
making the representation unitary. For proofs of these statements see e.g.\! \cite{KR87}.

\subsection{Lowest weight representations of the \texorpdfstring{$\cW_3$}{W3}-algebra}\label{lowest}

Given a representation of the $\cW_3$-algebra $\{L_n,W_n\}_{n\in\ZZ}$ with central charge $c$, a nonzero vector $\Omega_{c,h,w}=:\Omega$ is said to be a {\bf lowest weight vector} with lowest weight $(h,w) \in \CC^2$, if 
\begin{equation}\label{eq:weights}
\text{for all } n>0: L_n\Omega = W_n\Omega = 0, \;\;\text{and}
\; L_0\Omega = h\Omega, \; W_0 = w\Omega.
\end{equation}
In case $h=w=0$, $\Omega$ is said to be a {\bf vacuum vector}. In case the lowest weight vector is cyclic, the whole representation is 
said to be a {\bf lowest weight representation}.

Using the $\cW_3$-algebra relations, it is rather easy
(however, the induction should go with respect to $g$ in Appendix \ref{verma} instead of the number
of operators, see e.g.\! \cite{BMP96})
to show that for any lowest weight representation, the vectors of the form
\begin{align}\label{eq:vector}
 L_{-m_1}\cdots L_{-m_\ell}W_{-n_1}\cdots W_{-n_k}\Omega,
\end{align}
where $1 \le m_1\le \cdots \le m_\ell, 1 \le n_1\le \cdots \le n_k$, span the whole representation space.
However, in general, linear independence does not follow. 
Nevertheless, for each central charge $c \neq -\frac{22}5$ and lowest weight $(h,w) \in \CC^2$ there is indeed a representation, the {\bf Verma module} $\underline{V}^{\cW_3}_{c,h,w}$, where 
these vectors form a basis. It is rather clear that such a representation is essentially unique;
what is less evident, is its existence.
For a Lie algebra, Verma modules are constructed as a quotient of the universal covering algebra,
see e.g.\! \cite{Jacobson79}.
As the $\cW_3$-algebra is not a Lie algebra and the commutator $[W_m, W_n]$ contains an infinite sum
in $L$'s, it is actually nontrivial that Verma modules exist. We show this in a novel, analytic manner in Appendix \ref{verma}.

Using the $\cW_3$-algebra relations, it is not difficult to see that the Verma module can admit at most one
invariant Hermitian form $\<\cdot, \cdot\>$ 
with normalization $\<\Omega_{c,h,w}, \Omega_{c,h,w}\> = 1$. We will call this the ``canonical'' form. 
It is also rather trivial that if $c,h,w$ are not all real, then such a Hermitian form cannot exists.
Again, what is less evident is the existence for $c,h,w \in \RR$. We give a proof of this fact in Appendix \ref{verma}.
Since the goal of this paper is to deal with unitarity, we will focus on the case when 
$c,h,w\in \RR$. 

Let us now take some $c,h,w\in \RR$, $c\neq -\frac{22}{5}$.
Any nontrivial subrepresentation in the Verma module is included in the kernel
\[
 \ker \<\cdot,\cdot\> = \{\Psi \in \underline{V}^{\cW_3}_{c,h,w}: \<\Psi, \Phi\> = 0
\text{ for all }\Phi \in \underline{V}^{\cW_3}_{c,h,w}\},
\]
see the arguments of \cite[Proposition 3.4(c)]{KR87}.
It then turns out that with the given values of 
$c,h,w$, there is (an up-to-isomorphism) unique irreducible lowest weight representation $V^{\cW_3}_{c,h,w}$:
namely, the one obtained by taking
the quotient of the Verma module with respect to $\ker \<\cdot,\cdot\>$. The canonical form on a Verma module is positive semidefinite
if and only if the corresponding irreducible representation admits a invariant inner
product, making it unitary. 

Actually, standard arguments show  that (for given $(c,h,w)$) \emph{any} lowest weight representation with a non-degenerate, invariant Hermitian form $\<\cdot,\cdot\>$ is isomorphic
to the unique irreducible representation. This is due to the fact that the value of 
$\<\Psi,\Psi'\>$, where $\Psi,\Psi'$ are vectors of the form \eqref{eq:vector}, is
``universal'': it depends on $c,h,w$ but \emph{not} on the actual representation;  
see Proposition \ref{pr:invariant}.
In particular, for each triplet $(c,h,w)$, there is (up to isomorphism) at most
one lowest weight representation with an invariant inner product; namely, $V^{\cW_3}_{c,h,w}$.

\subsection{The Kac determinant}
\label{Kacdet}
The question of when the canonical form $\<\cdot,\cdot\>$ on the Verma module $\underline{V}^{\cW_3}_{c,h,w}$
is degenerate or positive semidefinite can be studied through the Kac determinant.
See \cite[Chapter 8]{KR87} for an overview of the methods used here, which are written for
infinite-dimensional Lie algebras, but apply to the $\cW_3$-algebra as well.

The Hermitian form $\<\cdot,\cdot\>$ vanishes on pairs of vectors of the form \eqref{eq:vector}
when the eigenvalue $N= \sum_j m_j + \sum_j n_j$
 of $L_0$ are different, hence the question can be studied for each $N \ge 0$ separately.
There are finite many vectors $\Psi^{(N)}_1,\ldots \Psi^{(N)}_{d_N}$ among \eqref{eq:vector} for any given $N$
that span a finite dimensional subspace in $\underline{V}^{\cW_3}_{c,h,w}$ and one can consider
the Gram matrix $M_{N,c,h,w}$ whose entries are the product values $\<\Psi^{(N)}_j,\Psi^{(N)}_k\>$. 
Note that these values are real polynomials of $c,\frac{1}{22+5c},h,w$ (see Appendix \ref{verma}).
Evidently, we have the following.
\begin{itemize}
 \item $\underline{V}^{\cW_3}_{c,h,w}$ is irreducible if and only if all of these matrices are nondegenerate.
 
 \item The canonical form on $\underline{V}^{\cW_3}_{c,h,w}$ is
 positive (semi)definite if and only if these matrices are all positive (semi)definite.

\end{itemize}
However, it is difficult to determine the rank and positive (semi)definiteness of all these matrices at once.
Nevertheless, a rather compact formula can be given for the determinant $\det(M_{N,c,h,w})$ at level $N$ 
-- called the {\bf Kac determinant} -- of these matrices. We can use it in the following ways.
\begin{itemize}
 \item If $\underline{V}^{\cW_3}_{c,h,w}$ is reducible, then $\det(M_{N,c,h,w}) = 0$ for some $N$.
 \item If the canonical form on $\underline{V}^{\cW_3}_{c,h,w}$ is
 positive-definite, then $\det(M_{N,c,h,w}) > 0$ for all $N$.
\end{itemize}
At each level $N$, $\det(M_{N,c,h,w})$ is a polynomial of $c,\frac{1}{22+5c},h,w$.
Therefore, if one finds a vector in $\ker \<\cdot,\cdot\>$ in a Verma module $\underline{V}^{\cW_3}_{c,h,w}$,
one can extract a factor from $\det(M_{N,c,h,w})$ for some $N$.
With sufficiently many such vectors in $\ker \<\cdot,\cdot\>$, one can determine $\det(M_{N,c,h,w})$
up to a multiplicative positive constant. 
According to \cite{Mizoguchi89, ACHMP18}, the Kac determinant at level $N$ is
\[
 \det(M_{N,c,h,w})\sim \prod_{k=1}^N \prod_{mn=k} (f_{mn}(h,c) - w^2)^{P_2(N-k)},
\]
where ``$\sim$'' means equality up to a positive multiplicative constant that can depend on $N$
(but not on $c,h,w$) and
\[
 \sum_{n=0}^\infty P_2(n)t^n = \prod_{n=1}^\infty \frac1{(1-t^n)^2}
\]
and
\begin{align}\label{eq:fmn}
 f_{mn}(h,c) &= \frac{64}{9(5c+22)} \left[h + (4-n^2)\a_+^2 + (4-m^2)\a_-^2 - 2 + \frac{mn}2\right] \nonumber\\
 &\qquad\qquad\qquad \times \left[h - 4((n^2-1)\a_+^2 + (m^2-1)\a_-^2) - 2(1-mn)\right]^2
\end{align}
with 
\begin{align*}
\a_\pm^2 = \frac{50-c \pm \sqrt{(2-c)(98-c)}}{192}.
\end{align*}
We shall exploit the knowledge of the signs of the Kac determinant (given by these explicit formulas) in two ways: 
\begin{itemize}
\item
 Let $H \subset \RR^3$ be a connected set where for any $(c,h,w)\in H$ and any $N\in\NN$ it holds that
 $\det(M_{N,c,h,w})> 0$. In this situation, if $V^{\cW_3}_{c',h',w'} = \underline{V}^{\cW_3}_{c',h',w'}$
 is unitary for at least one triple $(c',h',w')\in H$, then it is so for all triples in 
 the closure $\overline{H}$.
 \item If $\det(M_{N,c,h,w}) < 0$ for some $N\in\NN$, then $\underline{V}^{\cW_3}_{c,h,w}$
 is not unitary.
\end{itemize}

By the observation of \cite[(A.10)]{ACHMP18},
if $2 < c < 98$, 
the contributions from $f_{mn}$ with $m\neq n$ are non-zero positive because $\a_\pm$ in \eqref{eq:fmn}
have non-zero imaginary parts, and since
\begin{align*}
 f_{mm}(c,h) = \frac{((c-2)m^2 - c + 24h +2)^2(96h + (c-2)(m^2 -4))}{7776(5c+22)}
\end{align*}
is increasing with respect to $m$, hence all Kac determinants are positive if
\begin{align}\label{eq:f11}
 f_{11}(h,c)-w^2 = \frac{h^2(96h-3(c-2))}{27(5c+22)} - w^2> 0.
\end{align}
Note that regardless of the value of the central charge,  $f_{11}(h,c)-w^2 \ge 0$ is a necessary condition for unitarity
since $f_{11}(h,c)-w^2$ is the first Kac determinant up to a positive constant.

The case $h=0$ is of particular importance, as this is when the lowest weight vector is a ``vacuum vector'' for the Virasoro subalgebra.
From the observation above, unitarity together with $h=0$ implies $w=0$.

\section{Unitarity of lowest weight representations}\label{unitarity}

\subsection{The free field realization of Fateev and Zamolodchikov 
}
\label{free}
Given a pair of commuting fields $a_{[k]}(z)=\sum_{n\in \ZZ} J_{[k],n}z^{-n-1}$ $(k=1,2)$, both satisfying the $\uone$-current relation (\ref{eq:aa}),
one can construct a family of representations
of the $\cW_3$-algebra 
depending on a complex parameter $\alpha_0$.
Following Fateev and Zamolodchikov \cite{FZ87},
we set
\begin{align*}
 \tilde L(z;\a_0) 
 &= 
      L_{[1]}(z) + \sqrt 2\a_0\partial\aone(z) + L_{[2]}(z)
 \nonumber \\
 &=
  \frac12: \aone(z)^2:  + \frac12: \atwo(z)^2: + \sqrt 2\a_0\partial\aone(z),\\
 \tilde W(z;\a_0) 
 &= \frac{b}{12i} \big[i2\sqrt 2: \atwo(z)^3:  - i6\sqrt 2 :\aone(z)^2:\atwo(z) 
 -i6\a_0 \partial\aone(z)\atwo(z) 
 \nonumber \\
 &\qquad\qquad - i18\a_0 \aone(z)\partial\atwo(z) -i6\sqrt 2\a_0^2\partial^2\atwo(z)\big] 
 .
\end{align*}
\begin{theorem}[\cite{FZ87}]\label{th:fz}
Let $\a_0\in \CC$ be such that 
$c(\a_0):=2-24 \a_0^2 \neq -\frac{22}{5}$
and $b\in \CC$ such that $b^2 = \frac{16}{22+5c(\alpha_0)}$. Then the
above defined $ \tilde L(z;\a_0) , \tilde W(z;\a_0)$
fields satisfy the $\cW_3$-algebra relations \eqref{eq:commfield} with central charge $c(\a_0)$.
\end{theorem}

\begin{remark}
We think it useful to make some comments on the computations justifying the above theorem.
First of all, instead of commutation relations, it is more common to work in terms of
{\bf operator product expansions (OPEs)}. The OPE of two fields $F_1(z), F_2(z)$
is usually written in the form
\[
 F_1(z)F_2(\zeta) \sim \sum_{j=1}^N \frac{G_j(\zeta)}{(z-\zeta)^j},
\]
where $G_j(z), j = 1, \cdots, N$ are some other fields.
As formal  series, this relation should be interpreted as (see \cite[Theorem 2.3]{Kac98})
$ [F_1(z), F_2(\zeta)] = \sum_{j=1}^N \frac1{j!}\partial^j_\zeta \delta(z-\zeta)G_j(\zeta).$

It is possible to write the OPE between a field $F(z)$ and a normal product $: G(z)H(z):$
in terms of the OPE between $F,G,H$
and the fields appearing in their OPE; again, for details we refer to \cite{Kac98}. Thus, if the OPE algebra of the basic fields
is closed -- like in our case: $[\aone(z),\atwo(\zeta)]=0$ and $[a_{[j]}(z),a_{[j]}(\zeta)]=\partial_\zeta \delta(z-\zeta)$ ($j=1,2$) -- then in principle the OPE of any pair of normal products can be determined
in terms of the basic fields.
Therefore, Theorem \ref{th:fz} can be indeed proved only in terms of the commutation relation \eqref{eq:aa}.
Although actual computations of OPE of composite fields can be tedious and painful,
these computations are fortunately very established procedures and can be carried out by computers, too.
The most widely used software for this the Mathematica package\footnote{Mathematica scripts can be also executed on the freely download-able Wolfram Script; see more at \url{https://www.wolfram.com/wolframscript/}.}
OPEdefs \cite{Thielemans91} by Thielemans (although  
 there are also other packages, e.g.\! \cite{Ekstrand11}). As is indicated in the text, the authors of
\cite{RSW18} also used this package to make computations with OPEs related to the free-field
realizations of the $\cW$-algebras, and this is what we also used\footnote{We thank Simon Wood for providing us his own code he used for the computations in \cite{RSW18}.} in part to have an independent verification and in part to check that our constants (which, due to differing conventions, slightly differ from the one appearing in \cite{FZ87}) are indeed rightly set.
\end{remark}

Since we are interested by unitarity, it is worth rewriting our fields using the circle derivative
$F'(z)= iz\partial_z F(z)$ and performing computations with the ``shifted fields'' we introduced above.
Also, we prefer to make some different choices of variables -- e.g.\! instead of $\alpha_0$ as in the previous
theorem, we will use $\kappa:= -i\sqrt{2}\alpha_0$ -- so that in the unitary case we will need to deal with real constants, only. We thought it useful for the reader to summarize our conventions in a table (which are actually
mainly the ones used by Buchholz and Schulz-Mirbach in \cite{BS90} and hence will be referred as the ``B-SM conventions'') and put it in contrast with the one used by the physicist and the one used by the VOA community.
\begin{table}[ht]
\label{conv_table}
\begin{tabular}{|l|l|l|}
\hline
Physicist  & VOA & B-SM / ours \\
\hline
$\varphi(z)$ (the massless free field) & undefined & undefined \\
$i\partial_z \varphi(z)$ & $\sqrt 2 a(z)$ & $\sqrt 2 J(z)/z$ \\
$T(z)$ & $L(z)$ & $T(z)/z^2 $ \\
$W(z)$ & $W(z)$ & $M(z)/z^3$ \\
$i\partial_z^2 \varphi(z)$ & $\sqrt 2\partial_z a(z)$ & $ -\sqrt2 (J(z)+iJ'(z))/z^2$ \\
$i\partial_z^3 \varphi(z)$ & $\sqrt 2\partial^2_z a(z)$
 & $(2\sqrt 2 J(z) + i3\sqrt2 J'(z) - \sqrt2 J''(z))/z^3$ \\
$-\frac14 : (\partial_z \varphi(z))^2: $ & $\frac12: a(z)^2: $ & $ : J^2:\!(z)/(2z^2) $\\
$:\partial_z \varphi(z)^n:$ & $(-i\sqrt{2})^n: a(z)^n:$ & $(-i\sqrt{2}/z)^n: J^n:\!(z) $ \\
$\sqrt 2\a_0$ & $\sqrt 2\a_0$ & $i\k$ \\
\hline
\end{tabular}
    \caption{Correspondence between fields and constants in various conventions.}
    \label{tb:translation}  
\end{table}

With $: J_{[j]}^n:\!(z) =  z^n :a_{[j]}(z)^n:$ $(j=1,2)$, we find that in the  Fateev-Zamolodchikov construction,
the fields $\tilde T(z;\k) :=  z^2\tilde{L}(z;\a_0)$ and
$\tilde M(z;\k) :=  z^3\tilde{W}(z;\a_0)$ can be written in the following way:
\begin{align}
 \tilde T(z;\k) 
      &= \frac12: \Jone^2:\!(z) - i\k (\Jone(z) + i\Jone'(z))+ \frac12: \Jtwo^2:\!(z), \label{eq:BSM-T}\\
 \tilde M(z;\k) 
 &= \frac b{3\sqrt 2}: \Jtwo^3:\!(z)  - \frac b{\sqrt 2} (: \Jone^2:\!(z) - i2\k(\Jone(z) + i\Jone'(z)))\Jtwo(z) \nonumber \\
 &\qquad + \frac{3b\k}{2\sqrt 2} (\Jone'(z) \Jtwo(z) - \Jone(z)\Jtwo'(z)) \nonumber \\
 &\qquad + \frac{b\k^2}{2\sqrt 2} (2 \Jtwo(z) + i3 \Jtwo'(z) -  \Jtwo''(z)).
 \label{eq:BSM-W}
\end{align}

Assume that $\Jone(z), \Jtwo(z)$ have a common lowest weight vector $\Omega_{q_1,q_2}$
with lowest weights $q_1, q_2$.
It is straightforward to check that $\Omega_{q_1,q_2}$
is annihilated by all positive
Fourier modes of fields like $:\Jtwo^3:\!(z)$ or $\Jone'(z)\Jtwo(z)$ and hence also 
by those of $\tilde T(z;\k)=\sum_{n\in\ZZ} \tilde{L}_{\k,n}z^{-n}$ and 
$\tilde M(z;\k)=\sum_{n\in\ZZ} \tilde{W}_{\k,n}z^{-n} $. One also computes that
\begin{align*}
\tilde{L}_{\k,0}\Omega_{q_1,q_2} &= \left(\frac{1}{2}a^2_{[1],0} + \frac{1}{2}a^2_{[2],0}
- i\kappa a_{[1],0}\right)\Omega_{q_1,q_2},\\
\tilde{W}_{\k,0}\Omega_{q_1,q_2} &= \frac{b}{\sqrt{2}} \left(
\frac{1}{3}a_{[2],0}^3 -(a_{[1],0}^2-2i\kappa a_{[1],0}) a_{[2],0} + \k^2 a_{[2],0}
\right)\Omega_{q_1,q_2}.
\end{align*}
Hence we have the following.
\begin{proposition}
\label{pr:lwvectors}
If $\Omega_{q_1,q_2}$ is a lowest weight vector for the two commuting $\uone$-currents $\Jone(z),\Jtwo(z)$ with corresponding lowest weights $q_1$ and $q_2$, respectively, then it is also a lowest weight vector for the representation of the $\cW_3$-algebra given by the fields
\eqref{eq:BSM-T} and \eqref{eq:BSM-W} with central charge $c = 2-24\a_0^2 = 2 + 12 \k^2$ and lowest weight $(h,w)$ where 
\[
h=\frac12 q_1^2+\frac12 q_2^2 - i\kappa q_1,\;\;
w=
\frac{b}{\sqrt{2}}\left(\frac{1}{3}q_{2}^3 -(q_{1}^2-2i\kappa q_{1} )q_{2} + \k^2 q_{2}\right)
\]
\end{proposition}

Now suppose we have an inner product on our representation space making
the currents $J_{[j]}(z)=z a_{[j]}(z)$ ($j=1,2$) symmetric. Then 
the fields $: \Jone^2:\!(z), J_{[1]}(z),J'_{[1]}(z),: \Jtwo^2:\!(z)$ are all symmetric, but the linear combination giving $\tilde{T}(z;\kappa)$ is only symmetric for $\kappa =0$; i.e. for the central charge $c=2$ case (and we have the same situation regarding $M(z)$).

One possible remedy would be a modification of our inner product; instead of the invariant form for our currents,
we should try to use a ``strange'' one that does not make
$J_{[1]}(z), J_{[2]}(z)$ symmetric. Here we will follow a -- in some sense -- dual approach. Namely, we retain our original inner product, but instead modify our currents by applying an
automorphisms of the algebra \eqref{eq:aa}.

\subsection{New representations by automorphisms of the
\texorpdfstring{$\uone$}{U(1)}-current
}\label{auto}

Suppose the field $J(z)=\sum_{n\in \ZZ}a_n z^{-n}$ is a $\uone$-current and 
$f(z)=\sum_{n\in \ZZ}c_n z^{-n}$ is a scalar valued field (i.e. $c_n=0$ for $n$ large enough).
Then, because scalars commute with everything, the sum $J(z)+f(z)$ satisfies the same commutation relation
of the $\uone$-current field. In terms of Fourier modes, the transformation is 
$a_n\mapsto a_n + c_n$. 
If further $c_n=0$ for all $n>0$ and $\Psi$ is a lowest weight vector for $J(z)$ with weight $q$
(i.e. we have $a_n\Psi=0$ for all $n>0$ and $a_0\Psi=q \Psi$), then 
$\Psi$ is a lowest weight vector for $J(z)+f(z)$ with lowest weight $q+c_0$.
Representations of this kind play a central role in \cite{BMT88}.

Evidently, the map $a_n\mapsto a_n + c_n$ can be interpreted as a composition of a representation with an automorphism of our Lie algebra. Thus, if we further used our current to construct something -- say a stress-energy field -- then by composition with such an automorphism, we get a ``transformed'' stress-energy field.
As an expression involving only normal powers and derivatives of $J(z) + f(z)$,
it still satisfies the same commutation relations with the same central charge,
because the latter relations are determined by the $\uone$ commutation relation.

Following the ideas of Buchholz and Schulz-Mirbach \cite[(4.6)]{BS90}, we consider the above transformation with $f(z) = \kappa \rho(z) + \eta$, where $\kappa,\eta$ are scalar constants and $\rho(z) = -i\frac{z-1}{z+1}$. As was explained in Section \ref{formal}, 
here we interpret $\rho(z)$ as the formal series (\ref{eq:rho}), rather than a function. 
Accordingly, $\rho_n=0$ for all $n>0$ and in terms of Fourier modes, our transformation 
is
\begin{align*}
 a_n &\longmapsto \varphi_{\k,\eta}(a_n) = a_n + i\k(\delta_{n,0} + 2(-1)^n \chi_{(-\infty,0)}(n)) + \eta \delta_{n,0},
\end{align*}
where $\chi_{(-\infty,0)}$ is the characteristic function of the open interval $(-\infty,0)$.

The reader might wonder what is the reason behind the choice of the function $\rho$. As we shall 
see in the next subsection, what makes $\rho(z)$ important is that it is a solution of the differential equation
\begin{equation}
\label{eq:diffeq}
\frac{\rho(z)^2}2 + \frac12 - \rho'(z) = 0,
\end{equation}
where $\rho'(z)$ denotes the derivative along the circle \eqref{eq:derivative}.

The transformed $\uone$-current field gives rise to a new associated stress-energy field. By an abuse of notation, we denote (the shifted version of) this by $\varphi_{\k,\eta}(T(z))$, even though $\varphi_{\k,\eta}$ does not formally act on $T(z)$. After a straightforward computation, we find that 
\begin{align*}
 \varphi_{\k,\eta}(T(z)) &:= \frac12 : \varphi_{\k,\eta}(J)^2:\!(z) 
 = T(z) + (\k \rho(z)+\eta)J(z) + \frac{(\k\rho(z)+\eta)^2}2,
\end{align*}
where $T(z) = \frac12 : J^2:\!(z) $ is the canonical stress-energy field of
the original representation.

\paragraph{``Almost'' symmetric stress-energy tensor with $c>1$.}
Following the work of Buchholz and Schulz-Mirbach, given a $\uone$-current field $J(z)$, apart from the canonical (shifted) stress-energy field $T(z) = \frac12 : J(z)^2:$,
 we shall also consider $T_\kappa(z)=\sum_{n\in\ZZ}L_{\kappa, n}z^{-n}$
where
\begin{equation}\label{eq:T_k}
 T_\k(z) = T(z) + \k\left(J'(z) - \rho(z)J(z)\right)
\end{equation}
and of course the product $\rho(z)J(z)$ is understood in the sense of fields;
i.e.\! its coefficient of $z^{-n}$ is
$\sum_m i\k(\delta_{m,0} + 2(-1)^m\chi_{(-\infty,0)}(m)) J_{n-m}$. Note that $T_0(z)=T(z)$; i.e.\! for
$\k=0$ the construction reduces to the canonical one.
One can show that the operators $\{L_{\kappa,n}\}_{\{n\in\ZZ\}}$ form a representation of the Virasoro algebra with central charge $c = 1+12\k^2$ by a straightforward computation. However, we will not need that since
we see this below in another way.

The representation \eqref{eq:T_k} is different from \eqref{constrKR87}:
the construction \eqref{constrKR87} does not yield a manifestly
unitary vacuum representation with central charge $c>1$. On the other hand,
if $0\neq \kappa\in\RR$ then $c>1$ and if $J(z)$ is symmetric and
$\Omega$ is a lowest weight vector for $J(z)$ with zero lowest weight $q=0$
(i.e.\! if $\Omega$ was a vacuum vector for $J(z)$),
then -- as is easily checked -- $\Omega$ is still a vacuum vector
for the representation $\{L_{\kappa,n}\}_{\{n\in\ZZ\}}$
($\Omega$ is not necessarily cyclic for $\{L_{\kappa,n}\}_{\{n\in\ZZ\}}$, even if it was so for $J(z)$).
Moreover, even if it is not properly symmetric, $T_\kappa(z)$ has a certain weakened symmetry property. Since the fields
$T(z), J(z), J'(z)$ appearing in our formula are symmetric, $\kappa\in\RR$ and $\rho$ is also real on the unit circle -- as was explained at the end of Section \ref{formal_adjoint} -- we have that 
\[
(p(z)T_\kappa(z))^\dagger = \overline{p}(z) T_\kappa(z)
\]
for any (scalar valued) trigonometric  polynomial 
$p(z)=\sum_{|n|<N}c_n z^{-n}$ satisfying the additional property
$p(-1)=0$.

Although different, this construction is closely related to \eqref{constrKR87}.
Indeed, if we apply the construction \eqref{eq:T_k}
to the current $\varphi_{\k,\eta}(J(z))$ instead of $J(z)$
(i.e.\! we apply the transformation $\varphi_{\k,\eta}$ with the same $\kappa$) then we obtain
the stress-energy field of \eqref{constrKR87}:
\begin{align}\label{eq:virasoro-nonunitary}
 \varphi_{\k,\eta}(T_\k(z)) &=T_\k(z) + (\k \rho(z)+\eta)J(z) + \frac{(\k\rho(z)+\eta)^2}2
 + \k\left(\k \rho'(z) - \k\rho(z)^2 -\eta\rho(z) \right) \nonumber \\
 &= T_\k(z) + (\k \rho(z)+\eta)J(z) + \frac{\k^2+\eta^2}2 \nonumber \\
 &= T_0(z) + \k J'(z) +\eta J(z) + \frac{\k^2+\eta^2}2,
\end{align}
where we used that $\rho(z)$ satisfies the differential equation \eqref{eq:diffeq}. 
This also shows that the operators $\{L_{\kappa,n}\}_{\{n\in\ZZ\}}$ indeed satisfy
the Virasoro relations with central charge $c=1+12\kappa^2$,
since the last expression coincides with \eqref{constrKR87}.

\paragraph{Restoring unitarity to the Fateev-Zamolodchikov realization}\label{restoring}

The transformation $\varphi_{-\k, i\k}$ will be of special interest. Since $\rho_0=i$, 
it changes the lowest weight value for $J(z)$ by $-i\k+i\k=0$; i.e.\! it preserves the lowest weight.
Moreover, by substituting $\eta=i\k$ in \eqref{eq:virasoro-nonunitary} and taking account of
the fact that $\varphi_{-\k, i\k}=\varphi_{\k,-i\k}^{-1}$, we see that
\begin{align}
 &\varphi_{-\k, i\k}(J(z)) = J(z) - \k\rho(z) + i\k, \nonumber \\
 &\varphi_{-\k, i\k}(J'(z)) = J'(z) - \k\rho'(z), \nonumber \\
 &\varphi_{-\k, i\k}(T(z)) := \frac12 : \varphi_{-\k, i\k}(J(z))^2:\; =
  T(z) + (-\k \rho(z)+ i\k)J(z) + \frac{(-\k\rho(z)+i\k)^2}2 \nonumber \\
 &\varphi_{-\k, i\k}(T(z) - i\k(J(z) +iJ'(z))) = T_\k(z) \label{eq:phikk}
\end{align}
suggesting that by applying $\varphi_{-\k, i\k}$ to the first of our
commuting currents appearing in the Fateev-Zamolodchikov construction, we could turn our
``very much non symmetric''  fields into ones that have a discussed weak form of symmetry
\emph{without} changing lowest weight values.

So let us take again two commuting $\uone$-current fields $\Jone(z), \Jtwo(z)$
and consider them as a representation of the direct sum of the Heisenberg algebra with itself. 
Then letting $\varphi_{-\k, i\k}$ act on the first one while not doing anything with the second one, i.e.\! the transformation $\tilde \varphi_{-\k,i\k}$ defined by
\[
 \tilde \varphi_{-\k,i\k}(\Jone(z)) = \varphi_{-\k, i\k}(\Jone(z)),\quad \tilde \varphi_{-\k,i\k}(\Jtwo(z)) = \Jtwo(z)
\]
can be viewed as a composition of our representation with an automorphism.
Accordingly, we can apply the Fateev-Zamolodchikov realization \eqref{eq:BSM-T}\eqref{eq:BSM-W}
to these representations $\tilde \varphi_{-\k,i\k}(\Jone(z)), \tilde \varphi_{-\k,i\k}(\Jtwo(z))$
and obtain a shifted pair of fields, which we denote by $\tilde T(z;\k)$ and $\tilde M(z;\k)$.
Setting $T_{[j],\k}(z) = \frac12 : J^2_{[j]}:\!(z) + \k\left(J_{[j]}'(z) - \rho(z)J_{[j]}(z)\right)$ as in \eqref{eq:T_k} for $j=1,2$, by a straightforward computation 
we find that 
\begin{align}
 \tilde \varphi_{-\k, i\k}(\tilde T(z;\k)) &= T_{[1],\k}(z) + \Ttwozero(z), \nonumber \\
 \tilde \varphi_{-\k, i\k}(\tilde M(z;\k))
 &= \frac b{3\sqrt 2}: \Jtwo^3:\!(z)  - \sqrt 2 b T_{[1],\k}(z)\Jtwo(z) \nonumber \\
 &\qquad + \frac{3b\k}{2\sqrt 2} ((\Jone'(z) - \k\rho'(z))\Jtwo(z)) - (\Jone(z) -\k\rho(z))\Jtwo'(z)). \nonumber \\
 &\qquad + \frac{b\k^2}{2\sqrt 2} (2 \Jtwo(z) -  \Jtwo''(z)) \label{eq:phikkTW}
\end{align}
Since we obtained them by a transformation which is in fact a composition with an automorphism
of a pair of $\uone$-currents, 
the fields $z^2\tilde \varphi_{-\k, i\k}(\tilde T(z;\k)), z^3\tilde \varphi_{-\k, i\k}(\tilde M(z;\k))$
must still result in a representation of the $\cW_3$-algebra. Moreover, since $\tilde \varphi_{-\k, i\k}$
transforms our currents in a manner that leaves every lowest weight vector a lowest weight vector with the same
weight, by Proposition \ref{pr:lwvectors}, we have that if $\Omega_{q_1,q_2}$ was a common lowest weight vector
for $J_{[1]}(z)$ and
$J_{[2]}(z)$ with lowest weights $q_1$ and $q_2$ respectively, then it will be also a lowest weight vector for
the representation of the $\cW_3$-algebra given by \eqref{eq:phikkTW} with lowest weight value $(h,w)$ given by
Proposition \ref{pr:lwvectors}. 
\begin{corollary}
Let $\kappa,q_1,q_2,b\in\RR$ be such that $b^2=\frac{16}{22+5c}$ where $c=2+12\k^2$. Then there exists a lowest weight representation $\{(L_n,W_n)\}_{n\in\ZZ}$ of the $\cW_3$-algebra with central charge $c=2+12\k^2$ and lowest weight $(h,w)=(\frac12 q_1^2+ \frac12 q_2^2-i\k q_1,
\frac{b}{\sqrt{2}}(\frac{1}{3}q_{2}^3 -(q_{1}^2-2i\kappa q_{1} )q_{2} + \k^2 q_{2}))$ on an inner product space such that
the fields $T(z)=\sum_{n\in\ZZ}L_n z^{-n}$ and 
$M(z)=\sum_{n\in\ZZ}W_n z^{-n}$ satisfy the weak symmetry condition
\[
(p(z)T(z))^\dagger = \overline{p}(z)T(z), \;\;\;
(r(z)W(z))^\dagger = \overline{r}(z)W(z)
\]
for all trigonometric polynomials $p,r$ with $p(-1)=r(-1)=r'(-1)=0$.
\label{cor:R-hermitian}
\end{corollary}
\begin{proof}
By taking a tensor product of two lowest weight representations, it is clear that we can construct 
two commuting symmetric $\uone$-current fields $J_{[1]}(z),J_{[2]}(z)$ on an inner product space
having a common lowest weight vector $\Omega_{q_1,q_2}$
of lowest weight $q_1$ and $q_2$, respectively. (Note: this is the point where we use that $q_1,q_2$ are real: with a nonzero imaginary part, we could not have an invariant inner product for our currents). Now consider the representation $z^2 T(z), z^3 M(z)$
of the $\cW_3$-algebra constructed 
through \eqref{eq:phikkTW} with the help of the fields $J_{[1]}(z)$ and $J_{[2]}(z)$. Taking account of the
symmetry of our currents, the fact that $\k,b\in\RR$ and the comments at the end of Section \ref{formal_adjoint}, we see that $T(z)$ and $M(z)$ indeed satisfy the required symmetry condition.
Moreover, by Proposition \ref{pr:lwvectors} and the observation above the current corollary,
$\Omega_{q_1,q_2}$ is a lowest weight vector for this representation with the claimed lowest weight value. Thus, restricting our representation of the $\cW_3$-algebra to the cyclic subspace of $\Omega_{q_1,q_2}$ gives a lowest weight representation with all the desired properties.
\end{proof}

\begin{remark}
One might wonder whether our ``weak'' symmetry condition in the above corollary actually implies ``true'' symmetry. It turns out that in the vacuum case this is exactly what happens -- we shall see this in the next section. However, note that in general, the
answer is: ``no''. In fact, if $q_1\neq 0$, then $h$ is not real, so we cannot even have an invariant Hermitian form (let alone an inner product). Actually, by \eqref{eq:f11}, even if we set $q_1=0$ (and hence have real $h$ and $w$), in general we cannot have unitarity (see Theorem \ref{th:unitarity} for
some values of $h,w$ for which unitarity fails).
Indeed, our argument in the next section will use in a crucial way that $h=w=0$.
In contrast, in the Virasoro case, the ``weak'' symmetry can indeed be turned into ``true'' one;
see Proposition \ref{prop_virR}.
\end{remark}

\subsection{Proof of unitarity for \texorpdfstring{$h=w=0$}{h=w=0}}\label{proof}

In this section we will work in an abstract setting: we suppose that $\{(L_n,W_n)\}_{\{n\in \ZZ\}}$ form a representation of the $\cW_3$-algebra with central charge $c\geq 2$ and that we are also given a nonzero vector $\Omega$ as well as an inner product $\<\cdot,\cdot\>$ satisfying the following requirements:
\begin{itemize}
\item[(i)] $\Omega$ is a cyclic lowest weight vector for our representation and $L_0\Omega=W_0\Omega=0$,
\item[(ii)] the fields $T(z)=\sum_{n\in\ZZ}L_n z^{-n}$ and 
$M(z)=\sum_{n\in\ZZ}W_n z^{-n}$ satisfy the condition
\[
(p(z)T(z))^\dagger = \overline{p}(z)T(z), \;\;\;
(r(z)M(z))^\dagger = \overline{r}(z)M(z)
\]
for all trigonometric polynomials $p,r$ with $p(-1)=r(-1)=r'(-1)=0$ (where
the adjoint is considered w.r.t.\! the given inner product $\<\cdot,\cdot\>$). 
\end{itemize}
Such a representation and inner product indeed exists; this is clear by
considering Corollary \ref{cor:R-hermitian} with $q_1=q_2=0$ and $\k=\sqrt{\frac{c-2}{12}}$. 
From now on we shall not be interested how these objects were explicitly constructed; we will
only use to above listed properties. Our aim will be to conclude that $\<\cdot,\cdot\>$ is in fact 
an invariant inner product for our representation, making it unitary.
Since we work with Fourier modes rather than fields, we begin with reformulating property (ii).

\begin{lemma}
Let $p(z)=\sum_{|n|<N}c_n z^{-n}$ and $r(z)=\sum_{|n|<N}d_n z^{-n}$ be a trigonometric polynomials 
satisfying $p(-1)=r(-1)=r'(-1)=0$. Then 
\[
\left(\sum_{|n|<N}c_{-n} L_{n}\right)^\dagger = \sum_{|n|<N}\overline{c}_{n} L_{-n}\;\;\;
{\rm and}\;\; 
\left(\sum_{|n|<N}d_{-n} W_{n}\right)^\dagger = \sum_{|n|<N}\overline{d}_n W_{-n}
\]
\end{lemma}
\begin{proof}
This is evident by considering the zero mode of the products appearing in 
the equalities of property (ii).
\end{proof}
This implies in particular that
$(L_{n_1} - (-1)^{n_1-n_2}L_{n_2})^\dagger = L_{-n_1} - (-1)^{n_1-n_2}L_{-n_2}$
and that for any $n_1,n_2,n_3\in\ZZ$ with $n_1 \neq n_2\neq n_3 \neq n_1$, with
a unique pair of real numbers $u,d\in \RR$ satisfying
$(-1)^{n_1} + (-1)^{n_2}u + (-1)^{n_3}d = (-1)^{n_1}n_1 + (-1)^{n_2}n_2u + (-1)^{n_3}n_3 d = 0$
(such a pair exists) we have  
\begin{align}
(W_{n_1} + u W_{n_2} + d W_{n_3})^\dagger
 = W_{-n_1} + u W_{-n_2} + d W_{-n_3}.
\label{eq:W+}
\end{align}
The next Lemma follows from Assumption (ii), and the form $\<\cdot,\cdot\>$ is not necessarily
the canonical one for $(c,h,w)$.
\begin{lemma}\label{lm:zerovector}
$L_{-1}\Omega = W_{-1}\Omega = W_{-2}\Omega=0$.
\end{lemma}
\begin{proof}
By now we know that $(L_{-1}+L_0)^\dagger = (L_{1}+L_0) =:A$; hence
\[
\|L_{-1}\Omega\|^2 = \|(L_{-1}+L_0)\Omega\|^2 = \<\Omega,A (L_{-1}+L_0)\Omega\>
\]
However, $A=(L_{1}+L_0)$ annihilates $\Omega$ so
\[
A(L_{-1}+L_0)\Omega =[A, (L_{-1}+L_0)]\Omega = (L_{1}+2L_0+L_{-1})\Omega =
(A+A^\dagger)\Omega = A^\dagger \Omega,
\]
and $\|L_{-1}\Omega\|^2 = \<\Omega,A^\dagger\Omega\> = 
\<A\Omega,\Omega\> = 0$ showing that $L_{-1}\Omega=0$. Then to conclude the proof it is enough to note that $W_{-1}=-\frac{1}{2}[L_{-1},W_0]$ and $W_{-2}=-[L_{-1},W_{-1}]$.
\end{proof}

\begin{theorem}\label{th:unitarity}
Let $\{(L_n,W_n)\}_{\{n\in \ZZ\}}$ be a representation of the $\cW_3$-algebra
with a scalar product $\<\cdot, \cdot \>$
satisfying Assumptions (i),(ii). Then 
$L_n^\dagger =  L_{-n}$ and $W_n^\dagger =  W_{-n}$ for all $n\in\ZZ$. 
Consequently, the representation is unitary.
\end{theorem}
\begin{proof}
 We fist show that $L_0^\dagger = L_0$. Each vector of the form \eqref{eq:vector} 
 is an eigenvector of $L_0$ with a real eigenvalue and since
 we are in a lowest weight representation, these vectors -- and hence also the eigenspaces of $L_0$ -- span the full space. So to prove that $L_0^\dagger = L_0$, it is enough to check that these vectors are orthogonal to each other whenever the associated eigenvalues of $L_0$ are not equal. We will do this by performing an induction. 

 \paragraph{Induction on $\mathrm{gr}:=2\times$(number of $L$'s) +  3$\times$(number of $W$'s).}
Assume that for some $j\in\{0,1,\ldots \}$ it holds that whenever $\ell,\ell',k,k'$ are nonnegative integers of ``total grade value'' (see \cite{BMP96} for a similar grading)
\[
 \mathrm{gr}:= 2(\ell+\ell') + 3(k+k')\leq j,
\]
then for any positive integers $m_1,\ldots m_\ell,n_1\ldots n_k$ and $m'_1,\ldots m'_{\ell'},n'_1\ldots n'_{k'}$, the vectors
\begin{align}
\nonumber
\Psi&=   L_{-m_1}\cdots L_{-m_\ell}W_{-n_1}\cdots W_{-n_k}\Omega \\
\Psi'&=   L_{-m'_1}\cdots L_{-m'_{\ell'}}W_{-n'_1}\cdots W_{-n'_{k'}}\Omega
 \label{eq:vectpair}
 \end{align}
are orthogonal unless $\lambda=\lambda'$ where $\lambda=m_1 + \cdots m_\ell + n_1 + \cdots n_{k_W}$ and $ \lambda'= m'_1 + \cdots m'_{\ell'} + n'_1 + \cdots n'_{k'}$ (i.e.\! unless they correspond to the same eigenvalue of $L_0$). Note that for $\mathrm{gr}=0$, our assumption is trivially true as in that case we have a single possible pair of vectors only: $\Psi=\Psi'=\Omega$. We have to show that this remains true for
$\mathrm{gr}=j+1$. We will do this by 
considering all possible pairs of vectors $\Psi,\Psi'$ of the form \eqref{eq:vectpair}
with $\mathrm{gr}=2(\ell+\ell') + 3(k+k')\leq j+1$ and
show that if $\lambda\neq\lambda'$, then $\<\Psi,\Psi'\>=0$.

 \paragraph{Case 1: $\mathbf{\ell+\ell'>0}.$}
 
If $\lambda=\lambda'$, there is nothing to prove, so assume $\lambda\neq\lambda'$. 
Since  now we treat the case when the sum of $\ell$ and $\ell'$ is positive, at least one of them must be nonzero; so say $\ell\geq 1$, meaning that $\Psi$  must contain at least one $L$ operator in its defining expression.
 Let then $\xi$ be the vector obtained by removing the first $L$ operator from the expression of $\Psi$,
 namely, $\Psi=L_{-m_1}\xi$. The vector $\xi$ is still given by an expression of the 
 same form than $\Psi$ or $\Psi'$, but the corresponding eigenvalue of $L_0$ is $(\lambda-m_1)$ and hence $\Psi=L_{-m_1}\xi =  (L_{-m_1}-(-1)^{m_1}L_0)\xi + (-1)^{m_1}(\lambda-m_1)\xi$.
 By assumption (ii), 
 we know that $(L_{-m_1}-(-1)^{m_1}L_0)^\dagger =(L_{m_1}-(-1)^{m_1}L_0)$. Putting all this together, 
 we have
 \begin{align*}
 \<\Psi,\Psi'\> = \< L_{-m_1}\xi,\Psi'\> &= 
 \<(L_{-m_1}-(-1)^{m_1}L_0)\xi + (-1)^{m_1}(\lambda-m_1)\xi,\Psi\>\\
 &=
 \<\xi,(L_{m_1}-(-1)^{m_1}L_0)\Psi'\>+
 (-1)^{m_1}(\lambda-m_1)\<\xi,\Psi'\>\\
 &=\<\xi,L_{m_1}\Psi'\> + (-1)^{m_1}(\lambda-\lambda'-m_1)
 \<\xi, \Psi'\>
 \end{align*}
 We will argue that both terms in the above sum are separately zero and we begin with the second term.
 The total number of $L$'s in the expression giving $\xi$ and $\Psi'$ is $(\ell-1)+\ell'$ and the total number of $W$'s is $k+k'$. Thus, by the inductive
hypothesis, their inner product is zero unless they correspond to the same eigenvalue of $L_0$, in which case we must have $\lambda-m_1=\lambda'$. In either case, the product $(\lambda-\lambda'-m_1)\<\xi, \Psi'\>$ is zero.
Let us treat now the term $\<\xi,L_{m_1}\Psi'\>$. Since $L_{m_1}$ annihilates $\Omega$,  
\[
L_{m_1}\Psi'=L_{m_1} L\ldots L W\ldots W\Omega=[L_{m_1}, L\ldots L W\ldots W]\Omega,
\]
where we just symbolically wrote ``$L\ldots L W\ldots W$'' without detailing the indices.
Using the $\cW_3$-algebra commutation relations, the above vector can be rewritten as a linear combinations
of vectors of the form \eqref{eq:vectpair} with the same associated eigenvalue of $L_0$ -- i.e. with eigenvalue $\lambda'-m_1$ -- but with strictly smaller values of the quantity ``2 $\times$ the number of $L$'s +  3 $\times$ the number of $W$'s ''. (E.g.\! note that when exchanging the two $W$ operators, then, due to the commutation relations, two ``new'' $L$
operators can appear -- but only on the ``cost'' of having two $W$ operators less. This is why we gave more weight to a $W$ operator than an $L$ operator.) Therefore, again by the 
inductive hypothesis and $\l \neq \l'$, we have$\<\xi,L_{m_1}\Psi'\>=0$ and thus $ \<\Psi,\Psi'\>=0$.

 \paragraph{Case 2: $\mathbf{\ell=\ell'=0}.$}
 
 In this second case we have no $L$ operators at all in the defining expressions of our two vectors: $\Psi=W_{-n_1}\ldots W_{-n_k}\Omega$ and $\Psi'=W_{-n_1}\ldots W_{-n'_{k'}}\Omega$. Again we may assume that $\lambda\neq \lambda'$, and so in particular we must have at least one $W$ operator in our expressions (otherwise $\Psi=\Psi'=\Omega$). So say $k\geq 1$ and let $\xi$ be the vector obtain by removing the last $W$ from the expression of $\Psi$. Then $W_{-n_1}\xi=\Psi$ and $L_0\xi=(\lambda-n_1)\xi$.

By Lemma \ref{lm:zerovector}, $W_s\Omega=W_{-s}\Omega=0$ for $s\in \{0,1,2\}$. Since the index set $\{0,1,2\}$ has three elements, there must exists at least two different $r,s\in\{0,1,2\}$ such that neither $W_{-r}\xi$ nor $W_{-s}\xi$ does not correspond to the same eigenvalue of $L_0$ as $\Psi'$; i.e.\! that
$\lambda'\neq (\lambda-n_1+s), (\lambda-n_1+r)$. Then by \eqref{eq:W+}, we have some real numbers $u,d$ such that we have the adjoint relation $(W_{-n_1}+u W_{-r}+d W_{-s})^\dagger = W_{n_1}+u W_{r}+d W_{s}=:A$ holds, hence
\begin{align*}
 \<\Psi,\Psi'\> = \< W_{-n_1}\xi,\Psi'\> &= 
 \<(A^\dagger-(u W_{-r}+d W_{-s})) \xi,\Psi'\>\\
 &=
 \<\xi,A\Psi'\>\,-\,\<(u W_{-r}+d W_{-s})\xi,\Psi'\>.
 \end{align*}
 Since both $A=W_{n_1}+u W_{r}+d W_{s}$ and $B=(u W_{-r}+d W_{-s})$ annihilate $\Omega$, one can rewrite the above expressions using commutators as a linear combination of terms with strictly smaller total value of the quantity ``2 $\times$ the number of $L$'s +  3 $\times$ the number of $W$'s''
 than the original value $\mathrm{gr}$. Moreover, by our choice of $s$ and $r$, the corresponding eigenvalues of $L_0$ of the terms on the two sides of the inner product never coincide. So again by the inductive hypothesis, each of those inner product values are zero and hence $\Psi$ and $\Psi'$ are orthogonal.

Now we know that $L_0^\dagger=L_0$.
By Assumption (ii) we have $(L_n - (-1)^{n}L_0)^\dagger = L_{-n} -(-1)^{n} L_0$
and that $(W_{-1}+W_1-2W_0)^\dagger = (W_{-1}+W_1-2W_0)=:A$.
By taking real-linear combinations, we conclude then that $L_n^\dagger =  L_{-n}$ for all $n\in\ZZ$.
Then also $B:=i[L_0,A]=i(W_{-1}-W_1)$ is symmetric, and so is 
$C=i[L_0,B]=W_{-1}+W_1$ and $\frac12(C-A)=W_0$. We then have 
\[
 W_n^\dagger = \frac{1}{2n}[L_n,W_0]^\dagger = -\frac{1}{2n}[L_{n}^\dagger,W_0^\dagger] = 
-\frac{1}{2n}[L_{-n},W_0]=W_{-n}.
\]
\end{proof}

\begin{corollary}\label{cr:unitary}
 The irreducible lowest weight representation of the $\cW_3$-algebra with central charge $c \ge 2$
 and lowest weights $h=w=0$ is unitary.
\end{corollary}

\begin{remark}
By the existence theorem \cite[Theorem 4.5]{Kac98}, any lowest weight representation 
where the lowest weight vector $\Omega$ satisfies the extra condition $L_{-1}\Omega=0$, generates
a vertex algebra with translation operator $T=L_{-1}$. (This condition implies that the lowest weight must be $(h,w)=(0,0)$ but not the other way around. Note however that in the unitary case, $h=0$ alone implies $L_{-1}\Omega=0$.) This vertex algebra evidently has a Virasoro element $\nu=L_{-2}\Omega$ whose corresponding field has $T$ as a component, and since
the representation space is the direct sum of eigenspaces of $L_0$ with non-negative integral eigenvalues
and each eigenspace is finite dimensional as it is spanned by finite many vectors of the form \eqref{eq:vector}, the resulting structure is actually a vertex operator algebra (VOA). 
Moreover, if the representation we started with was unitary, then the obtained VOA is also unitary in the sense of \cite[Definition 5.2]{CKLW18}; see e.g.\cite[Proposition 5.17]{CKLW18}, which says that unitarity follows if the VOA is generated by a family of Hermitian\footnote{Note that Hermitianity of the fields $L(z)$ and $W(z)$ in the sense of \cite{CKLW18} is precisely equivalent to the symmetry of $T(z)=z^2L(z)$ and $M(z)=z^3W(z)$.} quasi-primary fields. The unitary VOAs constructed in this way must coincide with the simple quotients of the freely generated VOAs defined from the Verma modules for the $\cW_3$-algebra in \cite[Section 5]{Linshaw09}.
The latter can be identified as special cases of the universal VOAs in \cite{DK05, DK06},
as a consequence of \cite[Theorem 3.14]{DK06}, cf.\! also  \cite[Proposition 3.11, Example 3.14]{DK05}. 
\end{remark}

It is worth noting that with the same induction technique we used in this section, we can show that if a lowest weight representation of the Virasoro algebra $\{L_n\}_{n\in\ZZ}$ on an inner product space satisfies (ii) in the sense that $L_0-(-1)^nL_n = (L_0 -(-1)^n L_{-n})^\dagger$ for all $n\in\ZZ$, then in fact our inner product is an invariant form for the representation; in this case we do not need to assume that $h=0$.
 \begin{proposition}\label{prop_virR}
  Let $\{L_n\}_{\{n\in\ZZ\}}$ be a lowest weight representation of the Virasoro algebra
  with lowest weight $h\in\RR$ and lowest weight vector $\Omega$, and suppose that
  $(L_0 -(-1)^n L_{-n})^\dagger=L_0-(-1)^nL_n$ for all $n\in\ZZ$ with respect to a given Hermitian form $\<\cdot, \cdot\>$ (not necessarily the canonical one).
  Then $L_n = L_{-n}^\dagger$ for all $n\in\ZZ$.
 \end{proposition}
 \begin{proof}
  As in Theorem \ref{th:unitarity}, it is enough to prove that $L_0^\dagger = L_0$.
  Let $V_{h+n}$ be the eigenspaces of $L_0$.
  Assume that $V_h,\cdots,V_{n+h}$ are pairwise orthogonal. (For $n=0$ this is trivial.) 
  This implies that $L_0$ is symmetric
  when restricted to $V_h\oplus\cdots\oplus V_{h+n}$.
  Let $\xi \in V_{h+n+1}, \eta \in V_{k+h}, k \le n$.
  We have to show that $\<\xi, \eta\> = 0$.
  We may assume that $\xi = L_{-j}\zeta$, where $\zeta \in V_{h+n-j+1}$,
  as the general case is a linear combination.
  We have $(L_{-j} -(-1)^jL_0)^\dagger = L_{j} -(-1)^jL_0=:A$ and
  $L_{-j}=A^\dagger +(-1)^jL_0$, $L_{j}=A +(-1)^jL_0$ so 
  \begin{align*}
   \<\xi, \eta\> = \<L_{-j} \zeta, \eta\> &= \<(A^\dagger +(-1)^jL_0)\zeta, \eta\> 
   = \<\zeta, (A +(-1)^jL_0)\eta\> = \<\zeta, L_j\eta\> = 0,
  \end{align*}
  where the 3rd equality holds since $L_0$ is symmetric on $V_0\oplus\cdots\oplus V_{h+n}$,
  and the last equality follows from $L_j \eta \in V_h \oplus \cdots \oplus V_{h+n-j}$
  and the hypothesis of induction.
\end{proof}

\subsection{More unitary representations}\label{more}

It is also possible to construct unitary representations on the full space of the 
two commuting currents we used. Suppose again that we have two commuting 
$\uone$-current fields $J_{[j]}(z)=\sum_{n\in\ZZ}J_{n}z^{-n} = za_{[j]}(z)$ 
($j=1,2$) having a common lowest weight vector $\Omega_{q_1,q_2}$ with lowest weights
$q_1$ and $q_2$, respectively and that we have a fixed inner product on our representation space
making our currents $J_{[1]}(z), J_{[2]}(z)$ symmetric. Such currents on an inner product space indeed exist if $q_1,q_2\in \RR$ (e.g. consider the tensor product of two lowest weight representations). 
We now perform transformation $\tilde\varphi_{0,i\k}$; i.e.\! while remaining on the same inner product space, we consider the currents $\tilde\varphi_{0,i\k}(\Jone(z))=\Jone(z)+i\k,
\tilde\varphi_{0,i\k}(\Jtwo(z))=\Jtwo(z)$ instead of the original ones $\Jone(z),\Jtwo(z)$.
The vector $\Omega_{q_1,q_2}$ is still a common lowest weight vector for these currents, but this time with lowest weights $\tilde{q}_1=q_1+i\k$ and $\tilde{q}_2=q_2$. Recall that the transformation  $\tilde\varphi_{0,i\k}$ can be viewed as a composition of a representation with a Lie algebra automorphism, and can be further composed with the Fateev-Zamolodchikov realization of the $\cW_3$-algebra. By the usual abuse of notation, we shall denote the fields constructed from
$\tilde\varphi_{0,i\k}(\Jone(z))$ and $\tilde\varphi_{0,i\k}(\Jtwo(z))$ using the 
formulas \eqref{eq:BSM-T} and \eqref{eq:BSM-W} by  $\tilde\varphi_{0,i\kappa}(\tilde T(z;\k))$
and  $\tilde\varphi_{0,i\kappa}(\tilde M(z;\k))$. Note that $\k$ appears twice in these expressions: its value effects both the transformation we perform on the currents and the Fateev-Zamolodchikov
construction. In other words, we use the \emph{same} $\kappa$ value in both cases; this has the effect
that $\rho(z)$ vanishes from the formula: by a straightforward computation one has that
\begin{align*}
 \tilde \varphi_{0, i\kappa}(\tilde T(z;\k))
 &= T_{[1]}(z) + \k {\Jone}'(z) + \frac{\k^2}{2} + T_{[2]}(z),
 \nonumber
 \\
 \tilde \varphi_{0, i\kappa}(\tilde M(z;\k))
 &= 
\frac{b}{3\sqrt 2} : \Jtwo^3:\!(z) 
  - \sqrt 2 b \left(T_{[1]}(z) + \k {\Jone}'(z) + \frac{\k^2}{2}\right)\Jtwo(z) \nonumber \\
 &\qquad + \frac{3b\k}{2\sqrt 2} \Jone'(z)\Jtwo(z) -  \frac{3b\k}{2\sqrt 2}\Jone(z)\Jtwo'(z) 
 + \frac{b\k^2}{\sqrt 2} \Jtwo(z)  - \frac{b\k^2}{2\sqrt2} \Jtwo''(z).
\end{align*}
where of course $T_{[j]}(z)=\frac12: J^2_{[j]}:\!(z)$ $(j=1,2)$.

Note that $ \tilde \varphi_{0, i\kappa}(\tilde T(z;\k))$ is the sum of the following two stress-energy fields: the canonical one of the second current, and the modified one \eqref{constrKR87} -- with the constant $\eta=0$ -- of the first current. Note also that both $ \tilde \varphi_{0, i\kappa}(\tilde T(z;\k))$
and  $\tilde \varphi_{0, i\kappa}(\tilde M(z;\k))$ are manifestly symmetric if $\k$ is real: so for any 
 $\k \in \RR$, they give a unitary representation of the $\cW_3$-algebra with central charge $c=2+12\k^2$. 
Moreover, by Proposition \ref{pr:lwvectors}, $\Omega_{q_1,q_2}$ is a lowest weight vector for this representation, and the corresponding lowest weight $(h,w)$ can be obtained by replacing $q_1$ by $q_1+i\k$
in the formula of  Proposition \ref{pr:lwvectors}. After some simplifications, this gives 
\begin{align}\label{eq:morereps}
h=\frac{q_1^2+q_2^2+\k^2}{2},\;\;\;\;
w = b\left(\frac{q_2^3 - 3q_1^2 q_2}{3\sqrt 2} \right).
\end{align}
\begin{theorem}\label{th:2c}
 Let $c\geq 2$. By the above construction, the irreducible lowest weight representation of the $\cW_3$-algebra $V^{\cW_3}_{c,h,w}$ is unitary for
 \[
 h\geq \frac{c-2}{24},\;\; |w|\leq \sqrt{\frac{8}{198+45c}}
 \left(2h-\frac{c-2}{12}\right)^{\frac{3}{2}}\;\;\;\;(h,w\in\RR).
 \]
\end{theorem}
\begin{proof}
For each value of $(q_1, q_2, \k) \in \RR^3$, we have a unitary representaition with
central charge $c = 2+ 12\k^2$ and $(h,w)$ given by \eqref{eq:morereps}.
What we need to find out now is the set of $(c,h,w)$ values that can be realized in this manner.

It is clear that the possible values of $(c,h)$ are
$c \ge 2$ and $h \ge \frac{\k^2}2$.
We consider them (hence $\k,b$ and $q_1^2+q_2^2$) as given.
By varying $q_1, q_2$ under $q_1^2+q_2^2 = 2h - \frac{c-2}{12} =: C^2$ with $C\ge 0$, the function
\[
 q_2^3 - 3q_1^2 q_2 = 4q_2^3 - 3C^2q_2
\]
takes a local maximum at $q_2 = -\frac C 2$, hence the maximum under the condition $|q_2| \le C$
is $C^3$ (in both cases $q_2 = C, -\frac C 2$) and every value between them is possible.
Similarly, the minimum is $-C^3$.
This means that $|w| \le \sqrt{\frac{8}{198+45c}}
 \left(2h-\frac{c-2}{12}\right)^{\frac{3}{2}}$.
\end{proof}

For the convenience of the algebra-oriented reader, we show the ``unshifted'' fields:
\begin{align*}
 &\tilde \varphi_{0, i\kappa}(\tilde L(z;\a_0)) \\
 &:= \Lone(z) + \sqrt 2 \a_0 \partial\aone(z) + \sqrt 2\a_0 \aone(z)z^{-1} - \a_0^2 z^{-2}
 + \Ltwo(z) \\
 &\tilde \varphi_{0, i\kappa}(\tilde W(z;\a_0)) \\
 &:= \frac{b}{3\sqrt 2} \colon \atwo(z)^3\colon  - \sqrt 2 b\left(\Lone(z) + \sqrt 2\a_0 \aone(z)z^{-1}
 + \frac1{2\sqrt 2}\a_0 \partial\aone(z)+ \frac{\a_0^2 z^{-2}}2\right) \cdot \atwo(z) \\
 &\qquad - \frac{3b\a_0}{2} \aone(z) \cdot \partial\atwo(z) - \frac{b\a_0^2}{\sqrt2} (\partial^2 \atwo(z) + 3\partial \atwo(z)z^{-1}).
\end{align*}

This results allows us to completely characterize unitarity in the region $2\le c\le 98$.
\begin{corollary}\label{cr:2c98}
Let $2\le c \le 98$. Then the irreducible lowest weight representation of the $\cW_3$-algebra $V^{\cW_3}_{c,h,w}$ is unitary if and only if $f_{11}(h,c)-w^2 = \frac{h^2(96h-3(c-2))}{27(5c+22)} - w^2 \ge 0$.
\end{corollary}
\begin{proof}
 As we already mentioned at \eqref{eq:f11}, the condition $f_{11}(h,c)-w^2\geq 0$ is necessary for unitarity, so we only need to show the ``if'' part. Consider the open region $H$ and the closed region $R$ defined by
 \begin{align*}
H &=\{ (c,h,w)\in\RR^3 \, | \, 2<c<98, \, f_{11}(h,c)-w^2>0 \}, \\
R &=\{ (c,h,w)\in\RR^3 \, | \, 2\leq c\leq 98,\,  f_{11}(h,c)-w^2\geq 0 \}. 
 \end{align*}
Our aim is to prove unitarity in the region $R$. Now one that $R=\overline{H}\cup \{(c,0,0)|2\leq c\leq 98\}$ and Corollary \ref{cr:unitary} tells us that we indeed have unitarity on the line $\{(c,0,0)|2\leq c\leq 98\}$; so let us turn our attention to the region $\overline{H}$. 
 
It is clear that $f_{11}(h,c)$ is monotonically increasing with respect to $h$ and hence that $(c,h,w)\in H$ if and only if  $2<c<98, h > \frac{c-2}{32}, |w|<\sqrt{f_{11}(c,h)}$. In particular, $H$ is connected. As we already mentioned at \eqref{eq:f11}, in this region all Kac determinants are positive and hence, as was explained in Section \ref{Kacdet}, unitarity at a single point of $H$ implies unitarity for the entire closure $\overline{H}$. Since e.g.\! $(3,\frac{1}{24},0)\in H$, and at $c=3, h=\frac{1}{24}, w=0$ unitarity holds
by the previous theorem, therefore, we have unitarity on $\overline H$.
\end{proof}

\section{Outlook}\label{outlook}
The existence of unitary vacuum representations urges us to investigate
the conformal field theories (conformal nets and vertex operator algebras, see e.g.\! \cite{CKLW18})
related with these representations. Specifically, we are interested in the following questions.
\begin{itemize}
 \item Can one always construct a conformal net using the unitary vacuum representations?
 \item Are all other unitary representations associated with DHR sectors of these conformal nets? 
 (C.f.\! \cite{Carpi04, Weiner17} for the similar question regarding the Virasoro algebra.)
 \item How does the present result generalize to other $\cW$-algebras?
\end{itemize}

\subsubsection*{Acknowledgements}
 We thank Andrew Linshaw for discussions on Verma modules and Simon Wood for providing us with Mathematica codes. We also thank the Simons Center for Geometry and Physics -- the place where we started to write up this article -- and in particular the organizers of the program ``Operator Algebras and Quantum Physics'' for inviting us there.

S.C.\! and Y.T.\! also acknowledge the MIUR Excellence Department Project awarded to the Department of Mathematics, University of Rome Tor Vergata CUP E83C18000100006
and University of Rome ``Tor Vergata'' funding scheme ``Beyond Borders'' CUP E84I19002200005.

\appendix

\section{Lowest weight representations and Verma modules}\label{verma}
Since the $\cW_3$-algebra is not a Lie algebra, the notion and existence of Verma modules with invariant forms are not evident.
In physics literature they are either assumed without any further explanation \cite{BS93, Artamonov16} or claimed that they can be obtained -- in a similar manner to the Lie algebra case -- through the quotient of the ``universal covering algebra'' \cite{BMP96} which however cannot be constructed in the same way as in a Lie algebra because the commutation relation contains an infinite sum in terms of the basic fields. The more careful treatment of infinite sums at \cite[Section 5]{Linshaw09} might lead to a sensible construction, but the argument
as it is written there has the problem\footnote{We contacted the author who
indicated some possible remedies that might work; in any case, we shall not make use
of such a universal algebra.} that the ideal contains only finite sums, hence infinite sums cannot be reordered).
In \cite{DK05}, a Poincar\'e-Birkhoff-Witt type theorem is shown for $\cW$-algebras in general; however, 
it is in an abstract setting and it is not
evident for us whether it validates the particular form of Verma modules and invariant forms appearing both
in the physicist literature and also in our work. For these reasons, we decided to provide our own proof of these facts. 

Even if these results might be well-known to experts, the argument we give could be interesting on its own: instead of being
algebraic, in some sense it is analytic. We start with concrete constructions covering only some values of the central charge and lowest weights and then show
that all these objects -- e.g. the invariant form -- can be continued analytically to all values of the parameters.

A bilinear form $(\cdot,\cdot)$ is {\bf twisted-invariant} for a representation $\{L_n,W_n\}_{\{n\in\ZZ\}}$
of the $\cW_3$-algebra, if $(L_n x,y) = (x,L_{-n}y)$ and $(W_n x,y) = (x,W_{-n}y)$
for all $x,y$ vectors from the representation space and $n\in\ZZ$. The form is said to be
{\bf symmetric}, if $(x,y)=(y,x)$ for all $x,y$, and a symmetric form is
{\bf nondegenerate}, if ``$(x,y)=0$ for all $y$'' implies that $x=0$. Note that whereas in the main part of article we considered invariant Hermitian forms, to be more general, here we consider twisted-invariant bilinear forms. (It is not difficult to see that the existence of a nonzero invariant Hermitian form for a lowest weight representation rules out non-real lowest weights.)

To simplify notations, we set $K_{(X,n)}=L_n$ for $X=L$ and $K_{(X,n)}=W_n$ for $X=W$ and often write just $K_\nu$, with a shorthand notation $\nu = (X,n)$. We also set $-(X,n):=(X,-n)$ and further introduce a level $\l$ and another quantity $g$ 
by setting,
for every $r\in \{0,1,\cdots\}$ and 
and $\nu_1=(X_1,n_1),\cdots, \nu_r=(X_r,n_r)\in \{L,W\}\times \ZZ$,
\begin{align*}
 \l(\nu_1,\cdots ,\nu_r) &:= n_1 + \cdots + n_r  \\
 g(\nu_1,\cdots ,\nu_r) &:= d(X_1) + \cdots + d(X_r)
\end{align*}
where $d(L)=2$ and $d(W)=3$ (see \cite{BMP96} for a similar grading).
Note that both $g$ and $\l$ are completely symmetric in their arguments. 

Let $\{L_n,W_n\}_{n\in\ZZ}$ form a lowest weight representation of the
$\cW_3$-algebra with central charge $c\neq -\frac{22}5$, 
lowest weight $(h,w)\in \CC^2$ and
lowest weight vector $\Psi$.
Then,
using the $\cW_3$-algebra relations \eqref{eq:comm} and that $\Psi$ is a lowest weight
vector, it is straightforward to show that for any permutation $\sigma$, the difference 
\[
K_{\nu_1}\cdots K_{\nu_r}\Psi - 
K_{\nu_{\sigma(1)}}\cdots K_{\nu_{\sigma(r)}}\Psi
\]
can be written as a linear combination of terms of the form $K_{\nu_1'}\cdots K_{\nu_s'}\Psi$
with $g(\nu_1', \ldots, \nu_s')$ strictly smaller\footnote{This is exactly why we gave more ``weight'' to the $W$ operators by setting $d(W)=3>2=d(L)$ in the definition of $g$. We needed this because, roughly speaking, the commutator between two $W$ operators can give rise to two $L$ operators. The degree $d(\cdot)$ is defined so that it can be reduced using the commutation relations.} than $g(\nu_1, \ldots, \nu_s)$ and coefficients
which are real polynomials of $c,\frac{1}{22+5c}, h$ and $w$.
In particular, it follows that the cyclic space obtained from $\Psi$ is spanned by vectors of the form 
$K_{\nu_1}\cdots K_{\nu_r}\Psi$ where $r\in\{0,1,\cdots\}$, $\l(\nu_j)<0$ for each $j=1,\cdots r$ and $(\nu_1,\cdots ,\nu_r)$ is lexicographically ordered (namely, $\mu = (X_\mu,m) \prec \nu = (X_\nu,n)$ if $X_\mu = W, X_\nu=L$, or $X_\mu = Y_\nu$ and $m < n$).
However, this is not the only important conclusion one can draw.  
\begin{proposition}\label{pr:invariant}
For any $r,s\in\{0,1,\cdots \}$ and $\nu_1,\cdots \nu_s,\mu_1,\cdots \mu_r\in\{L,W\}\times \ZZ$,
there exists a real polynomial $p$ such that whenever $\{L_n,W_n\}$ is a representation of
the $\cW_3$-algebra with central charge $c\neq -\frac{22}{5}$ on a space $V$
with a twisted-invariant \emph{bilinear} form $(\cdot,\cdot)$ and lowest weight vector $\Psi\in V$ with lowest weights $(h,w)$ and 
$(\Psi,\Psi)=1$, then
\[
(K_{\nu_1}\cdots K_{\nu_r}\Psi,\, K_{\mu_1}\cdots K_{\mu_s}\Psi)
 = p(c, \textstyle{\frac{1}{22+5c}},h,w).
\]

\end{proposition}
\begin{proof}
We shall inductively construct such polynomials without any particular knowledge about the actual representation. It is enough to deal with the case $s=0$, since  by the invariance of the form,
we can put everything on one side:
\[
(K_{\nu_1}\cdots K_{\nu_r}\Psi,\, K_{\mu_1}\cdots K_{\mu_s}\Psi) = 
(K_{-\mu_s}\cdots K_{-\mu_1}K_{\nu_1}\cdots K_{\nu_r}\Psi,\,\Psi). 
\]
If further $r=0$, then the claim is trivially true, while for $r=1$, we have the expression $(K_{\nu_1}\Psi, \Psi) =(\Psi,K_{-\nu_1}\Psi)$, showing that it is zero unless
$\l(\nu_1)= 0$, in which vase it is $h$ when $\nu_1=(L,0)$ and $w$ when $\nu_1=(W,0)$.
Thus the claim is true for $g(\nu_1,\cdots,\nu_r)\leq 3$. 
Now assume the claim is true for $g(\nu_1,\cdots,\nu_r)< n $ and consider the case
$g(\nu_1,\cdots,\nu_r)=n>3$. If $\l(\nu_1)<0$, then the by moving $K_{\nu_1}$ to the other side, we see that $(K_{\nu_1}\cdots K_{\nu_r}\Psi,\Psi)=0$. If $\l(\nu_1)=0$, then 
by the same argument, the value of the form is $h (K_{\nu_2}\cdots K_{\nu_r}\Psi,\Psi)$
or $w(K_{\nu_2}\cdots K_{\nu_r}\Psi,\Psi)$, depending on whether $\nu_1=(L,0)$ or $(W,0)$.
In both cases we are done, as by the inductive hypothesis, we already have a polynomial giving the value of $(K_{\nu_2}\cdots K_{\nu_r}\Psi,\Psi)$. If finally $\l(\nu_1)>0$, then 
$K_{\nu_1}$ annihilates $\Psi$ and
\[
K_{\nu_1}K_{\nu_2}\cdots K_{\nu_r}\Psi = 
(K_{\nu_1}K_{\nu_2}\cdots K_{\nu_r} - 
K_{\nu_2}\cdots K_{\nu_r}K_{\nu_1})\Psi.
\]
which, as was mentioned, can be rewritten as a
linear combination of terms of the form $K_{\nu_1'}\cdots K_{\nu_s'}\Psi$
with $g(\nu_1', \ldots, \nu_s')$ strictly smaller than $g(\nu_1,\mu_1, \ldots, \nu_r)$ and coefficients which are real polynomials of $c,\frac{1}{22+5c}, h$ and $w$.
This concludes the induction.
\end{proof}
\begin{corollary}
The $\cW_3$-algebra admits a lowest weight representation with a 
symmetric, non-degenerate twisted-invariant bilinear form
form for every value of the central charge $c\neq -\frac{22}{5}$ and
lowest weight $(h,w)\in\CC^2$. If further $c,h,w\in \RR$, then the same remains true 
even if we replace the words ``symmetric bilinear'' by ``Hermitian''.
\end{corollary}
\begin{proof}
Consider a lowest weight representation with either a non-degenerate, symmetric twisted-invariant bilinear form $(\cdot,\cdot)$ or a non-degenerate Hermitian invariant sesquilinear form $\langle\cdot, \cdot\rangle$. If $c,h,w\in \RR$, then the arguments used in our previous proof remain valid regardless whether we apply them for $(\cdot,\cdot)$ or $\langle\cdot,\cdot\rangle$ and show that the product of elements from the \emph{real} subspace $M$ spanned by vectors of the form $K_{\nu_1}\cdots K_{\nu_r}\Psi$ is real and hence -- because of the non-degeneracy of the form -- that $M\cap iM=\{0\}.$
It then follows that starting from either $(\cdot,\cdot)$ or from $\langle\cdot, \cdot\rangle$, the equation
\[
 \langle a+ib, c+id \rangle = (a-ib, c+id)\;\;\;\; (a,b,c,d\in M)
\]
defines unambiguously the other object with all the desired properties. 

By the construction in Section \ref{more}, there exists a region $H\subset \RR^3$ with nonempty interior such that for all $(c,h,w)\in H$, there is a lowest weight representation of the $\cW_3$-algebra with central charge
$c$ and lowest weight $(h,w)$ having an invariant \emph{inner} product
(see Theorem \ref{th:2c} for an actual description of the region $H$).
In particular, for these values of $c,h$ and $w$ we also have the existence
of a non-degenerate, symmetric twisted-invariant bilinear form. Now suppose the value of $c\neq -\frac{22}{5}$, $h$
and $w$ are arbitrary. Let $\tilde V$ be the linear space freely spanned by (at the moment formal) expressions 
of the form $K_{\nu_1}\cdots K_{\nu_r}\Psi$ where $r\in\{0,1,\cdots\}$. We introduce a bilinear form 
on $\tilde V$ by setting
\[
(K_{\nu_1}\cdots K_{\nu_r}\Psi,\, K_{\mu_1}\cdots K_{\mu_s}\Psi)
= p(c,\textstyle{\frac{1}{22+5c}},h,w)
\]
where for each choice of $\nu_1, \cdots, \nu_r$ and $\mu_1, \cdots, \mu_s$, $p$ is a (possibly different) polynomial as in Proposition \ref{pr:invariant}. Note in particular, that the above value  given to the form is a rational function of $c,h,w$, and thus it is completely determined by its values in $H$.

To check that the introduced form is symmetric, we need to verify that 
\[
(K_{\nu_1}\cdots K_{\nu_r}\Psi,\, K_{\mu_1}\cdots K_{\mu_s}\Psi)
=
(K_{\mu_1}\cdots K_{\mu_s}\Psi,\, K_{\nu_1}\cdots K_{\nu_r}\Psi)
\]
for each choice of $\nu_1, \cdots, \nu_r$ and $\mu_1, \cdots, \mu_s$. However -- though
not indicated in notations -- each side of the above expression is a rational function of 
$c,h,w$, and when $(c,h,w)\in H$, we indeed have an equality. But if an equality of rational functions holds in $H$, then so does for all of their domain.

Let $V$ be the space obtained by factorizing $\tilde V$ with the set of ``null-vectors'',
i.e.\! by the subspace $\tilde N := \{x\in \tilde V:\text{ for all } y\in \tilde V: (x,y)=0\}$.
On this space, our form is still well-defined, symmetric, bilinear and by its construction, non-degenerate. We have to show that the natural action of the $K$ operators on $V$ is well-defined and gives a lowest weight representation of the $\cW_3$-algebra on the factorized space $V$.

To show well-definedness, we need to check that if $x \in \tilde N$, then $K_{\nu}x \in \tilde N$; that is, $(K_{\nu}x, y) = 0$ for all (non-commutative) polynomial $y$ in $\{L_n, W_n\}$.
We know that the left-hand side is a rational function of $(c,h,w)$ and that its value is indeed zero in $H$
-- and hence that it is zero on \emph{all} of its domain. This proves well-definedness.
Lastly, to verify that $V$ gives a lowest weight representations,
we only have to repeat the argument: both of the $\cW_3$ relations and the lowest weight property
are written as equalities between rational functions in $c,h,w$ with only singularity at $c = -\frac{22}5$,
therefore, their validity in $H$ implies their validity for all $(c,h,w), c \neq -\frac{22}5$.
\end{proof}

Although we do not need Verma modules for our main results, we think it worth explaining
how their existence can be verified
using reasoning similar to what we have just employed. In addition, although we will need Kac determinants and in particular
the results of Mizoguchi in \cite{Mizoguchi89}, we note that, for the notion of Kac determinant to be well-defined,
there is no need to have a Verma module. Indeed, as was explained,
the value of $(K_{\nu_1}\cdots K_{\nu_r}\Psi,\, K_{\mu_1}\cdots K_{\mu_s}\Psi)$ is \emph{universal}:
it depends only on the central charge $c$ and lowest weights $h,w$, but not the particular representation.
Indeed, to obtain his result, Mizoguchi never considers Verma modules; he works with some concrete representation
to find null-vectors. Therefore, our use in Corollary \ref{cr:2c98} and Proposition \ref{pr:verma}
of the Kac determinant computed in \cite{Mizoguchi89} does not involve circular arguments and is justified.

\begin{proposition}\label{pr:verma}
For every value of the central charge $c\neq -\frac{22}5$ and lowest weights $(h,w)\in\CC^2$, there exists (an up to isomorphism) unique lowest weight representation of the $\cW_3$-algebra with lowest weight vector $\Psi$
in which vectors of the form 
\begin{align}\label{eq:vectorsrev}
 L_{m_1}\cdots L_{m_r}W_{n_1}\cdots W_{n_s}\Psi
\end{align}
where $n_1\leq \cdots \leq n_s<0$ and $m_1\leq \cdots \leq m_r<0$, form a basis; 
i.e.\! a Verma representation.

This representation admits a unique twisted-invariant bilinear form $(\cdot,\cdot)$ with normalization $(\Psi,\Psi)=1$, and this form is automatically symmetric. Moreover, if in addition $c,h,w\in \RR$, then everything remains true even if we replace
the words ``bilinear'' by ``sesquilinear'' and ``symmetric'' by ``Hermitian''. 
\end{proposition}
\begin{proof}
By now we know that for every $c\neq -\frac{22}5$ and $(h,w)\in\CC^2$ there is an irreducible
lowest weight representation. However, in this representation, when $(c,h,w) \in H$,
where $H$ is the set introduced in the proof of Corollary \ref{cr:2c98},
the vectors \eqref{eq:vectorsrev} are independent (since in $H$ all Kac determinants are strictly positive)
and thus this representation is the Verma one.

For the rest of values,
we consider the abstract space $V$ spanned freely by vectors of the form \eqref{eq:vectorsrev}. By doing so, seemingly we have linear independence for free. However, we have to check that it carries a corresponding representation!
At this point, we use quotation marks and write 
symbols such as "$K_{\nu_1}\cdots K_{\nu_r}\Psi$", as this is indeed a vector of $V$ by construction,
but it is not (yet) the vector $\Psi$ acted on by $K$.
Given a $c\neq -\frac{22}5$ and $(h,w)\in\CC^2$, our task is then to define, for each $\nu$, an operator $K_\nu$ acting on $V$ so that they satisfy 
the following requirements:
\begin{itemize}
    \item[(i)] $K_\nu \Psi=0$ whenever $\l(\nu)>0$, $L_0\Psi=h\Psi$, $W_0\Psi=w\Psi$
    \item[(ii)] if $\nu,\nu_1\cdots, \nu_r$ are lexicographically ordered and $\ell(\nu),\ell(\nu_1),\ldots \ell(\nu_r)<0$, then the action of $K_\nu$ on the (abstract) vector "$K_{\nu_1}\cdots K_{\nu_r}\Psi$" should result in
the (abstract) vector "$K_{\nu}K_{\nu_1}\cdots K_{\nu_r}\Psi$". 
    \item[(iii)] $\{K_\nu\}_{\nu\in\{L,W\}\times \ZZ}$ is a representation of the $\cW_3$-algebra with central charge $c$.
\end{itemize}
Let us enumerate our basis vectors of the form \eqref{eq:vectorsrev} and denote them
by $\Psi_0 = \Psi,\Psi_1,\Psi_2,\ldots$. An action of $K_\nu$ can be defined
by fixing its matrix-components; i.e.\! by choosing 
scalars $M_{\nu, j,k}(c,h,w)\in \CC$ and setting $K_\nu \Psi_j := \sum_k M_{\nu, j,k}(c,h,w)\Psi_k$.
When $(c,h,w) \in H$, we know that this can be done in a way so that requirements (i), (ii) and (iii) are met, because for those values we do have Verma representations. However, it is not difficult to see that again, the coefficients $M_{\nu, j,k}(c,h,w)$ given by those Verma representations which are already known to exist, are rational expressions of the central
charge $c$ and lowest weights $(h,w)$ with real coefficients and possible singularity only at $c=-\frac{22}5$.
Thus, we can naturally continue them also outside of $H$.

We use these analytically continued matrix coefficients define the operators $K_\nu$.
Again, since inside $H$ these coefficients satisfy the properties (i), (ii) and (iii)
that are expressed in terms of rational functions of $c,h,w$ with only possible singularity at $c=-\frac{22}5$.
the same remains true outside.
This proves that we obtain a lowest weight representation on $V$.

\end{proof}

\newcommand{\etalchar}[1]{$^{#1}$}
\def\cprime{$'$} \def\polhk#1{\setbox0=\hbox{#1}{\ooalign{\hidewidth
  \lower1.5ex\hbox{`}\hidewidth\crcr\unhbox0}}} \def\cprime{$'$}

\end{document}